\documentclass[12pt,oneside,british,a4wide]{amsart}
\title{Jumps in speeds of hereditary properties in finite relational languages}
\author{Michael C. Laskowski}\thanks{The first author was partially supported
by NSF grants DMS-1308546 and DMS-1855789. The second author was partially supported by NSF grant DMS-1855711.}
\author{Caroline A. Terry}

\address{Department of Mathematics, University of Maryland, College Park,
MD 20742, USA}

\email{mcl@math.umd.edu}

\address{Department of Mathematics, The Ohio State University, Columbus,
OH 43210, USA}

\email{terry.376@osu.edu}

\usepackage{setspace,a4wide,color,xcolor,graphicx}

\usepackage{amssymb}
\usepackage{amsthm}
\usepackage{amsmath}
\usepackage{url}
\usepackage{enumerate}
\usepackage{comment}
\usepackage{color}
\def\showauthornotes{1}

\ifnum\showauthornotes=1
\newcommand{\Authornote}[2]{{\sf\small\color{magenta}{[#1: #2]}}}
\else
\newcommand{\Authornote}[2]{}
\fi

\def\abar{\overline{a}}
\def\bbar{\overline{b}}
\def\cbar{\overline{c}}
\def\dbar{\overline{d}}
\def\ebar{\overline{e}}
\def\hbar{\overline{h}}

\def\wbar{\overline{w}}
\def\xbar{\overline{x}}
\def\ybar{\overline{y}}
\def\zbar{\overline{z}}

\def\mbar{\overline{m}}

\def\phi{\varphi}
\def\calL{{\mathcal{L}}}
\def\calH{{\mathcal{H}}}
\def\calM{{\mathcal{M}}}
\def\calN{{\mathcal{N}}}
\def\calP{{\mathcal{P}}}
\def\calG{{\mathcal{G}}}
\def\calQ{{\mathcal{Q}}}

\def\calS{{\mathcal{S}}}
\def\calB{{\mathcal{B}}}
\def\calF{{\mathcal{F}}}

\def\calU{{\mathcal{U}}}

\def\calW{{\mathcal{W}}}

\def\FF{{\bf F}}

\def\Q{{\mathbb Q}}

\def\tp{{\rm tp}}
\def\qftp{{\rm qftp}}

\def\Fa0{{\FF^a_{\aleph_0}}}

\def\<{\langle}
\def\>{\rangle}
\def\o2{{^{\omega} 2}}

\def\n2{{^{n} 2}}

\def\tildeH{\widetilde{{\mathcal H}}}

\def\QMA{{\rm QMA}}
\def\Supp{{\rm Supp}}
\def\pp{{\bf p}}
\def\qq{{\bf q}}
\newtheorem{Theorem}{Theorem}[section]
\newtheorem{Proposition}[Theorem]{Proposition}
\newtheorem{Definition}[Theorem]{Definition}

\newtheorem{Remark}[Theorem]{Remark}
\newtheorem{Example}[Theorem]{Example}
\newtheorem{Lemma}[Theorem]{Lemma}
\newtheorem{Corollary}[Theorem]{Corollary}

\newtheorem{Fact}[Theorem]{Fact}
\newtheorem{Question}[Theorem]{Question}

\newtheorem{Problem}[Theorem]{Problem}
\newtheorem*{Definition4.2'}{Definition \ref{def:mam}'}

\def\lg{{\rm lg}}

\def\xhat{{\hat{x}}}

\newcommand\myrestriction{\mathord\restriction}
\def\mr#1{\myrestriction_{#1}}

\begin{document}

\date{February 3, 2022}

\begin{abstract}  
Given a finite relational language $\calL$, a \emph{hereditary $\calL$-property} is a class $\calH$ of finite $\calL$-structures closed under isomorphism and substructure.  The \emph{speed of $\calH$} is the function which sends an integer $n\geq 1$ to the number of distinct elements in $\calH$ with underlying set $\{1, . . . , n\}$.  In this paper we give a description of many new jumps in the possible speeds of a hereditary $\calL$-property, where $\calL$ is any finite relational language. In particular, we characterize the jumps in the polynomial and factorial ranges, and show they are essentially the same as in the case of graphs. The results in the factorial range are new for all examples requiring a language of arity greater than two, including the setting of hereditary properties of $k$-uniform hypergraphs for $k>2$. Further, adapting an example of Balogh, Bollob\'{a}s, and Weinreich, we show that for all $k\geq 2$, there are hereditary properties of $k$-uniform hypergraphs whose speeds oscillate between functions near the upper and lower bounds of the penultimate range, ruling out many natural functions as jumps in that range.  Our theorems about the factorial range use model theoretic tools related to the notion of mutual algebraicity.
\end{abstract}

\maketitle

\section{introduction}

A hereditary graph property, $\calH$, is a class of finite graphs which is closed under isomorphisms and induced subgraphs.  The \emph{speed} of $\calH$ is the function which sends a positive integer $n$ to $|\calH_n|$, where $\calH_n$ is the set of elements of $\calH$ with vertex set $[n]$.  Not just any function can occur as the speed of hereditary graph property.  Specifically, there are discrete ``jumps'' in the possible speeds. Study of these jumps began with work of Scheinerman and Zito in the 90's \cite{ZitoSch}, and culminated in a series of
papers from the 2000's by Balogh, Bollob\'{a}s, and Weinreich, which gave an almost complete picture of the jumps for hereditary graph properties. These results are summarized in the following theorem.

\begin{Theorem}[\cite{ABBM,BBW1, BBW2, ajump,BoTh2}]\label{thm:graphcase}
Suppose $\calH$ is a hereditary graph property.  Then one of the following holds.
\begin{enumerate}
\item There are $k\in \mathbb{N}$ and rational polynomials $p_1(x),\ldots, p_k(x)$ such that for sufficiently large $n$, $|\calH_n|=\sum_{i=1}^k p_i(n)i^n$.
\item There is an integer $k\geq 2$ such that $|\calH_n|=n^{(1-\frac{1}{k}+o(1))n}$.
\item There is an $\epsilon>0$ such that for sufficiently large $n$, $\mathcal{B}_n \leq |\calH_n|\leq 2^{n^{2-\epsilon}}$, where $\mathcal{B}_n\sim (n/\log n)^n$ denotes the $n$-th Bell number.
\item There is an integer $k\geq 2$ such that $|\calH_n|=2^{(1-\frac{1}{k}+o(1))n^2/2}$.
\end{enumerate}
\end{Theorem}

The jumps from (1) to (2) and within (2) are from \cite{BBW1}, the jump from (2) to (3) is from  \cite{BBW1, ajump}, the jump from (3) to (4) is from \cite{ABBM, BoTh2}, and the jumps within (4) are from \cite{BoTh2}.  Moreover, in \cite{BBW2}, Balogh, Bollob\'{a}s, and Weinreich showed that there exist hereditary graph properties whose speeds oscillate between functions near the lower and upper bound of range (3), which rules out most ``natural'' functions as possible jumps in that range.  Further, structural characterizations of the properties in ranges (1), (2) and (4) are given in \cite{BBW1} (ranges (1) and (2)) and \cite{BoTh2} (range (4)).

Despite the detailed understanding Theorem \ref{thm:graphcase} gives us about jumps in speeds of hereditary graph properties, relatively little was known about the jumps in speeds of hereditary properties of higher arity hypergraphs.  The goal of this paper is to generalize new aspects of Theorem \ref{thm:graphcase} to the setting of hereditary properties in arbitrary finite languages consisting of relations and/or constant symbols (we call such a language \emph{finite relational}). Specifically, we consider \emph{hereditary $\calL$-properties}, where, given a finite relational language $\calL$, a hereditary $\calL$-property is a class of finite $\calL$-structures closed under isomorphisms and substructures.  This notion encompasses most of the hereditary properties studied in the combinatorics literature\footnote{Notable exceptions include hereditary properties of permutations and non-uniform hypergraphs.}, including for example, hereditary properties of posets, of linear orders, and of $k$-uniform hypergraphs (ordered or unordered) for any $k\geq 2$.  We now summarize what was previously known about generalizing Theorem \ref{thm:graphcase} to hereditary properties of $\calL$-structures.

\begin{Theorem}[\cite{Alekseev1, BoTh1, meVCcount,TERRY2018}]\label{thm:hgcase}
Suppose $\calL$ is a finite relational language of arity $r\geq 1$ and $\calH$ is a hereditary $\calL$-property.  Then one of the following holds.
\begin{enumerate}[(i)]
\item There are constants $C, k\in \mathbb{N}^{>0}$ such that for sufficiently large $n$, $|\calH_n|\leq Cn^k$.\label{polgenweak}
\item There are constants $C,\epsilon>0$ such that for sufficiently large $n$, $2^{Cn}\leq |\calH_n|\leq 2^{n^{r-\epsilon}}$. \label{facgenweak}
\item There is a constant $C>0$ such that $|\calH_n|=2^{C{n\choose r}+o(n^r)}$.\label{topgenweak}
\end{enumerate}
\end{Theorem}

The existence of a jump to (\ref{topgenweak}) was first shown for $r$-uniform hypergraphs in  \cite{Alekseev1, BoTh1}, and later for finite relational languages in \cite{TERRY2018}. The jump between (\ref{facgenweak}) and (\ref{topgenweak}) as stated (an improvement from \cite{Alekseev1,BoTh1, TERRY2018}) is from \cite{meVCcount}.  The jump from (\ref{polgenweak}) and (\ref{facgenweak}) for finite relational languages is from \cite{meVCcount} (similar results were also obtained in \cite{FROSU}).  Stronger results, including an additional jump from the exponential to the factorial range, have also been shown in special cases (see \cite{balogh2007hereditary, BALOGH20061263, JGT:JGT20266, KaKl,KLAZAR2007258}). However, to our knowledge, Theorem \ref{thm:graphcase} encompasses all that was known in general, and even in the special case of hereditary properties of $r$-uniform hypergraphs for $r\geq 3$ (unordered).  Our focus in this paper is on the polynomial, exponential, and factorial ranges, where we obtain results analogous to those in Theorem \ref{thm:graphcase} for arbitrary hereditary $\calL$-properties. The arity of $\calL$ is the maximum arity of its relation symbols.  By convention, if $\calL$ consists of constant symbols, we say it has arity $0$.

\begin{Theorem}\label{thm:mainthm1}
Suppose $\calH$ is a hereditary $\calL$-property, where $\calL$ is a finite relational language. Then one of the following hold.
\begin{enumerate}
\item There are $k\in \mathbb{N}$ and rational polynomials $p_1(x),\ldots, p_k(x)$ such that for sufficiently large $n$, $|\calH_n|=\sum_{i=1}^kp_i(n)i^n$.
\item There is an integer $k\geq 2$ such that $|\calH_n|=n^{(1-\frac{1}{k}-o(1))n}$.
\item $|\calH_n|\geq n^{n(1-o(1))}$.
\end{enumerate}
\end{Theorem}

The most interesting and difficult parts of Theorem \ref{thm:mainthm1} are the jumps within range (2) and between ranges (2) and (3), which were not previously known for any hereditary property in a language of arity larger than two.  Combining Theorem \ref{thm:hgcase} with Theorem \ref{thm:mainthm1} yields the following overall result about jumps in speeds of hereditary $\calL$-properties.

\begin{Theorem}\label{thm:mainthm}
Suppose $\calH$ is a hereditary $\calL$-property, where $\calL$ is a finite relational language of arity $r\geq 0$. Then one of the following hold.
\begin{enumerate}
\item There are $k\in \mathbb{N}$ and rational polynomials $p_1(x),\ldots, p_k(x)$ such that for sufficiently large $n$, $|\calH_n|=\sum_{i=1}^kp_i(n)i^n$.
\item There is an integer $k\geq 2$ such that $|\calH_n|=n^{n(1-\frac{1}{k}-o(1))}$.
\item There is $\epsilon>0$ such that $n^{n(1-o(1))}\leq |\calH_n|\leq 2^{n^{r-\epsilon}}$.
\item There is a constant $C>0$ such that $|\calH_n|=2^{Cn^r+o(n^r)}$.
\end{enumerate}
\end{Theorem}

We also generalize results of \cite{BBW2} to show there are properties whose speeds oscillate between functions near the extremes of the penultimate range (case (3) of Theorem \ref{thm:mainthm}).  

\begin{Theorem}\label{thm:oscilate1}
For all integers $r\geq 2$,  and real numbers $c\geq 1/(r-1)$ and $\epsilon>1/c$, there is a hereditary property of $r$-uniform hypergraphs $\calH$ such that for arbitrarily large $n$, $|\calH_n|=n^{c n(r-1)(1-o(1))}$ and for arbitrarily large $n$, $|\calH_n|\geq 2^{n^{r-\epsilon}}$. 
\end{Theorem}

While there still remain many open problems about the penultimate range, Theorem \ref{thm:oscilate1} shows that, for instance, there are no jumps of the form $n^{kn}$ for $k>1$ or $2^{n^{k}}$ for $1<k<r$.

Together, Theorems \ref{thm:mainthm} and \ref{thm:oscilate1} give us a much more complete picture of the possible speeds of hereditary properties in arbitrary finite relational languages. In particular, our results show the possibilities are very close to those for hereditary graph properties from Theorem \ref{thm:graphcase}. Our proof of Theorem \ref{thm:mainthm} also gives structural characterizations of the properties in cases (1) and (2), and we will give explicit characterizations of the minimal properties with each speed in range (1).  Characterizations of the minimal properties in range (2) are more complicated and will appear in forthcoming work of the authors.  

The proofs in this paper owe much to the original proofs from the graphs setting, especially those appearing in \cite{BBW1, BBW2}.  However, a wider departure was required to deal with the jumps in the factorial range, namely case (2) of Theorem \ref{thm:mainthm1}.  Our arguments use the model theoretic notion of \emph{mutual algebraicity}, first defined by Laskowski in \cite{laskowski2013}.  A mutually algebraic property can be thought of as a generalization of a hereditary graph property of bounded degree graphs.  On the one hand, we will use the technology developed in \cite{Laskowski2009,laskowski2013} to obtain accurate estimates of $|\calH_n|$ when $\calH$ is a mutually algebraic property.  On the other hand, when $\calH$ is not mutually algebraic, we will use a theorem from \cite{ourpaper}, along with a new result, Proposition~\ref{prop:notMA}, to obtain an infinite model  $\calN$ of $T_{\calH}$ that has arbitrarily large finite substructures with enough isomorphic copies to witness that $|\calH_n|\ge  n^{n(1-o(1))}$.
Here and elsewhere, we use the compactness theorem, which allows us to work with infinite structures rather than large finite ones. 


The effectiveness of model theoretic tools in the context of this paper can be attributed to the fact that any hereditary $\calL$-property $\calH$ can be viewed as the class of finite models of a universal, first-order theory $T_{\calH}$. The speed of $\calH$ is then the same as the function sending $n$ to the number of distinct quantifier-free types in the variables $(x_1,\ldots, x_n)$ which are consistent with $T_{\calH}$, and which imply $x_i\neq x_j$ for distinct $i$ and $j$.  Problems about counting types have been fundamental to model theory for many years (see e.g. \cite{classification}). From this perspective it is not surprising that tools from model theory turn out to be useful for solving problems about speeds of hereditary $\calL$-properties.  Further, variations of this kind of problem have previously been investigated in model theory (see for example, \cite{ Mac:growth, LM:growth, macpherson_schmerl_1991}).  We will point out direct connections with this line of work throughout the paper.

We end this introduction by outlining some problems which remain open around this topic.  First, Theorem \ref{thm:graphcase} describes precisely the speeds occurring within the fastest growth rate (case (4)).  A similar analysis in the hypergraph setting would amount to understanding the possible Tur\'{a}n densities of hereditary hypergraph properties, a notoriously difficult question which we have made no attempt to address in this paper (see e.g. \cite{pikhurko}).  

There are many questions remaining around the penultimate range. For instance, in the graph case, Theorem \ref{thm:graphcase} gives a precise lower bound for the penultimate range, namely the $n$-th Bell number $\calB_n$.  This is accomplished in \cite{ajump} by characterizing the minimal graph properties in this range, which they show are the properties consisting of disjoint unions of cliques $\calH_{cl}$, or disjoint unions of anti-cliques, $\calH_{\overline{cl}}$.  A general analogue of this kind of result would be very interesting.

\begin{Problem}
Given a finite relational language $\calL$, characterize the minimal hereditary $\calL$-properties in the penultimate range.
\end{Problem}

It is easy to see the answer must be more complicated in general than the graph case.  Indeed, let $\calL=\{R(x,y)\}$ and consider the hereditary $\calL$-property $\calH$ consisting of all finite, transitive tournaments. Then $|\calH_n|=n!<\calB_n$ falls into the penultimate range.  Consequently, while $\calH_{cl}$, $\calH_{\overline{cl}}$ are the only minimal hereditary \emph{graph} properties in the penultimate range, there are other hereditary $\calL$-properties in this range with strictly smaller speed.

Theorem \ref{thm:oscilate1} rules out many possible jumps in the penultimate range, however, there does not yet exist a satisfying formalization of the idea that there can be no more ``reasonable'' jumps in this range (see  \cite{BBW2} for a thorough discussion of this).  Finally, while Theorem \ref{thm:oscilate1} shows there are properties whose speeds oscillate infinitely often between functions near the upper and lower bounds, there also exist properties whose speeds lies in the penultimate range, and for which wide oscillation is not possible (for an example of this, see \cite{BNS}). This leads to the following question.

\begin{Question} Suppose $\calL$ is a finite relational language.  Are there jumps within the penultimate range among restricted classes of hereditary $\calL$-properties (for instance among those which can be defined using finitely many forbidden configurations)?
\end{Question}

Both authors are grateful to the anonymous referee, who identified a gap in the proof of Proposition \ref{thm:malowerbound} in a prior version of this paper, and whose careful reading greatly improved the exposition.

\subsection{Notation and outline}

We now give an outline of the paper.  In Section \ref{sec:eu} we deal with the polynomial/exponential case, i.e. case (1) of Theorem \ref{thm:mainthm}.  Specifically, we define the class of \emph{basic} hereditary properties, and show the speed of any basic $\calH$ has the form appearing in case (1) of Theorem \ref{thm:mainthm}.  In Section \ref{sec:tb} we prove counting dichotomies for a restricted class of properties called \emph{totally bounded properties}, which generalize bounded degree graph properties.  In Section \ref{sec:ma} we define mutually algebraic properties, and prove counting dichotomies for mutually algebraic properties by showing they are controlled by finitely many totally bounded properties, after an appropriate change in language. We then show non-mutually algebraic properties fall into cases (3) or (4). In Section \ref{sec:oscillate}, we generalize an example from \cite{BBW2} to show that for all $r\geq 2$, there are hereditary properties of $r$-uniform hypergraphs whose speeds oscillate between functions near the upper and lower bounds of the penultimate range.

We spend the rest of this subsection fixing notation and definitions.  We have attempted to include sufficient information here so that the reader with a only a basic knowledge of first-order logic could read this paper.  

Suppose $\ell \geq 1$ is an integer, $X$ is a set, and $\xbar=(x_1,\ldots, x_{\ell})\in X^{\ell}$.  Then $[\ell]=\{1,\ldots, \ell\}$, $\cup \xbar =\{x_1,\ldots, x_{\ell}\}$, and $|\xbar|=\ell$.  We will sometimes abuse notation and write $\xbar$ instead of $\cup \xbar$ when it is clear from context what is meant. Given $\xbar=(x_1,\ldots, x_{\ell})$ and $I\subseteq [\ell]$, $\xbar_I$ is the tuple $(x_i:i\in I)$.  We write $\xbar\subseteq \ybar$ to denote that $\xbar$ is a subtuple of $\ybar$, i.e. $\xbar=\ybar_I$ for some $I$.  Given a sequence of variables $(z_1,\ldots, z_s)$, we write $\zbar=\xbar\wedge \ybar$ to mean there is a partition $I\cup J$ of $[s]$ into nonempty sets such that $\xbar=\zbar_I$ and $\ybar=\zbar_J$.  In this case, we call $\xbar\wedge \ybar$ a \emph{proper partition of $\zbar$}.  Set
$$
X^{\underline{\ell}}=\{(x_1,\ldots, x_{\ell}) \in X^{\ell}: x_i\neq x_j\text{ for each }i\neq j\}\quad \hbox{ and }\quad {X\choose \ell}=\{Y\subseteq X: |Y|=\ell\}.
$$
Notice that ${X\choose \ell}=\{\cup \xbar: \xbar\in X^{\underline{\ell}}\}$.  Given $u,v\in \mathbb{N}^{>0}$, a permutation $\sigma:[u]\rightarrow [u]$, and a set $\Sigma\subseteq [v]^u$ let 
$$
\sigma(\Sigma):=\{(v_{\sigma(1)},\ldots, v_{\sigma(u)}): (v_1,\ldots, v_u)\in \Sigma\}.
$$
We say $\Sigma$ is \emph{invariant under $\sigma$} when $\sigma(\Sigma)=\Sigma$.

We say a first order language $\calL$ is \emph{finite relational} if it consists of finitely many relation and constant symbols and no function symbols. Suppose $\calL$ is a finite relational language.  We let $|\calL|$ denote the total number of constants and relations in $\calL$.  By convention, the \emph{arity} of $\calL$ is $0$ if $\calL$ consists of only constant symbols, and is otherwise the largest arity of a relation in $\calL$.  Given an $\calL$-formula $\phi$ and a tuple of variables $\xbar$, we write $\phi(\xbar)$ to denote that the free variables of $\phi$ are all in the set $\cup \xbar$.  Similarly, if $p$ is a set of formulas, we write $p(\xbar)$ to mean every formula in $p$ has free variables in the set $\cup \xbar$. We will use script letters for $\calL$-structures and the corresponding non-script letters for their underlying set.  So for instance if $\calM$ is an $\calL$-structure, $M$ denotes the underlying set of $\calM$.  

Suppose $\calM$ is an $\calL$-structure. Given a formula $\phi(x_1,\ldots, x_s)$, 
$$
\phi^{\calM}=\{(m_1,\ldots, m_s)\in M^s: \calM\models \phi(m_1,\ldots, m_s)\}.
$$
A \emph{formula with parameters from $\calM$} is a an expression of the form $\phi(\xbar,\abar)$ where $\phi(\xbar,\ybar)$ is a formula and $\abar\in M^{|\ybar|}$.  The \emph{set of realizations of $\phi(\xbar,\abar)$ in $\calM$} is
$$
\phi(\calM,\abar):=\{\mbar\in M^{|\xbar|}: \calM\models \phi(\mbar,\abar)\}.
$$
Given $A\subseteq M$, a set $B\subseteq M^{|\xbar|}$ is \emph{defined by $\phi(\xbar,\ybar)$ over $A$} if there is $\abar\in A^{|\ybar|}$ such that $B=\phi(\calM; \abar)$.  In this case we say $B$ is \emph{definable in $\calM$}.  If $B$ is defined by a formula without parameters, we say that $B$ is \emph{$0$-definable}.  If $\Delta(x_1,\ldots, x_s)$ is a set of formulas (possibly with parameters from $M$), a \emph{realization of $\Delta$ in $\calM$} is a tuple $\mbar\in M^s$ such that $\calM\models \phi(\mbar)$ for all $\phi(x_1,\ldots, x_s)\in \Delta$.

If $c$ is a constant of $\calL$, recall that $c^\calM$ denotes the interpretation of $c$ in $\calM$.  Similarly, if $R$ is a $t$-ary relation of $\calL$, then $R^\calM:=\{\xbar\in M^t: \calM\models R(\xbar)\}$ is the interpretation of $R$ in $\calM$.  If $C$ is the set of constants of $\calL$, let $C^{\calM}=\{c^\calM: c\in C\}$. If $f:M\rightarrow N$ is a bijection, then $f(\calM)$ is the $\calL$-structure with domain $N$ such that $c^{f(\calM)}=f(c^{\calM})$ for each $c\in C$, and for each relation $R\in \calL$, $R^{f(\calM)}=f(R^{\calM})$.  An \emph{isomorphism from $\calM$ to $\calN$} is a bijection $f:M\rightarrow N$ such that $\calN=f(\calM)$.   An \emph{automorphism} of $\calM$ is an isomorphism from $\calM$ to $\calM$.

Given $X\subseteq M$ containing $C^{\calM}$, $\calM[X]$ is the $\mathcal{L}$-structure with domain $X$ such that for all $c\in C$, $c^{\calM[X]}=c^{\calM}$ and for all relations $R(x_1,\ldots, x_s)\in \calL$, $R^{\calM[X]}=R^{\calM}\cap X^s$.  An $\calL$-structure $\calN$ is an \emph{$\calL$-substructure of $\calM$}, denoted $\calN\subseteq_{\calL}\calM$ if and only if $\calN=\calM[X]$ for some $X\subseteq M$.  If $\calL$ is clear from context we will just write $\calN\subseteq \calM$.  The \emph{atomic formulas of $\calL$} are the formulas of the following forms.
\begin{itemize}
\item $t_1=t_2$ where each of $t_1,t_2$ is either a variable or a constant from $\calL$.
\item $R(t_1,\ldots, t_n)$, where $R(x_1,\ldots, x_n)$ is a relation symbol of $\calL$ and each $t_i$ is either a variable or a constant from $\calL$.  
\end{itemize}

The \emph{quantifier-free $\calL$-formulas} consist of all boolean combinations of atomic $\calL$-formulas.
When we write $\phi(x_1,\dots,x_n)$, we require the variable symbols $x_1,\dots,x_n$ be distinct.

\begin{Definition}  {\em  Suppose $\calM$ is a (possibly infinite) $\calL$-structure, $A\subseteq M$, and $\xbar$ is a subsequence of $\zbar$. 
For $\bbar\in M^{|\xbar|}$, 
$$\qftp(\bbar/A)=\{\hbox{quantifier-free formulas}\ \theta(\xbar,\abar): \calM\models\theta(\bbar,\abar)\ \hbox{and $\cup\abar\subseteq A$}\}$$
A {\em complete quantifier-free type $p(\xbar)$ over $A$} is anything of the form $\qftp(\bbar/A)$, where $\bbar\in N^{|\xbar|}$ for some $\calL$-structure $\calN\supseteq \calM$.
We let $S_{\xbar}(A)$ denote the set of complete quantifier-free types over $A$ in the variables $\xbar$.
}\end{Definition}

In model theory, it is more common to work with the notion of $\tp(\bbar/A)$, which denotes the set of all formulas (including those with quantifiers) over $A$ satisfied by $\bbar$.  However, as we
will be passing to finite substructures, we work exclusively with quantifier-free types.  Thus, any reference to ``types'' from here on out refers to quantifier-free types.

It will be convenient in Section \ref{sec:eu} to work with the following set of formulas.

 \begin{Definition}\label{def:neq}{\em 
 $$\Delta_{neq}:=\{\varphi(x_1,\ldots,x_s)\wedge \bigwedge_{1\leq i< j\leq s}x_i\neq x_j: \varphi(x_1,\ldots, x_s)\text{ is an atomic $\calL$-formula}\}.$$
}
\end{Definition}
For example, if $R(x,y)$ is a relation of $\calL$, then both $\tau(x_1,x_2)=R(x_1,x_2)\wedge x_1\neq x_2$ and $\phi(x)=R(x,x)$ and are in $\Delta_{neq}$.  Observe that for any $\calL$-structure $\calM$ and $\tau(x_1,\ldots, x_s)\in \Delta_{neq}$, $\tau^{\calM}\subseteq M^{\underline{s}}$.  Further, $\calM$ is completely determined by knowing $\tau^{\calM}$ for each $\tau\in \Delta_{neq}$.  Specifically, if $\calN$ is an $\calL$-structure satisfying $\tau^{\calN}=\tau^{\calM}$ for all $\tau\in \Delta_{neq}$, then $\calM=\calN$. 

A \emph{hereditary $\calL$-property} is a collection of finite $\calL$-structures which is closed under isomorphism and $\calL$-substructures. Every hereditary $\calL$-property $\calH$ can be axiomatized using a (usually incomplete) universal theory, which we denote by $T_{\calH}$. Specifically, for every hereditary $\calL$-property $\calH$, there is a set of universal sentences $T_{\calH}$, such that for any finite $\calL$-structure $\calM$, $\calM\in \calH$ if and only if $\calM\models T_{\calH}$.

A hereditary $\calL$-property $\calH$ is \emph{trivial} if there are only finitely many non-isomorphic $\calM\in \calH$.  Equivalently, $\calH$ is trivial if there is $N\in \mathbb{N}$ such that $\calH_n=\emptyset$ for all $n\geq N$.  Since we are interested in the size of $\calH_n$ for large $n$, we will be exclusively concerned with non-trivial hereditary $\calL$-properties in this paper.

\begin{Definition}{\em
Given an $\calL$-structure $\calM$, the \emph{universal theory of $\calM$}, $Th_{\forall}(\calM)$ is the set of sentences true in $\calM$ which are of the form $\forall x_1\ldots \forall x_n \phi(x_1,\ldots, x_n)$, where $\phi(x_1,\ldots, x_n)$ is a quantifier-free $\calL$-formula.

The \emph{age of $\calM$}, denoted $age(\calM)$, is the class of finite models of $Th_{\forall}(\calM)$. }
\end{Definition}

The age of $\calM$ is always a hereditary $\calL$-property (but not every hereditary $\calL$-property is the age of a single structure).  We will use throughout the paper the following standard model theoretic facts (see for instance Section 6.5 of \cite{hodges})
\begin{enumerate}
\item If $\calN\subseteq \calM$, and $\varphi$ is a universal sentence, then $\calM\models \varphi$ implies $\calN\models \varphi$.
\item If $\calN\models Th_{\forall}(\calM)$, then there is $\calM'\models Th(\calM)$ such that $\calN\subseteq \calM'$. 
\end{enumerate}
Together these imply that if $\calH=age(\calM)$, then for all $n\in \mathbb{N}$, $\calH_n$ is the set of all $\calL$-structures with domain $[n]$ and which are isomorphic to a substructure of $\calM$.  More generally, if $\calH$ is any hereditary property, and $\calM\models T_{\calH}$, then $age(\calM)\subseteq \calH$.

\section{Case 1: Polynomial/Exponential Growth}  \label{sec:eu}
In this section we give a sufficient condition for a hereditary property to have speed of the special form appearing in case (1) of Theorem \ref{thm:mainthm} (we will see later it is in fact necessary and sufficient).  Throughout this section, $\calL$ is a finite relational language, and $r\geq 0$ is the arity of $\calL$. We will use the following natural relation defined on any $\calL$-structure.

\begin{Definition}\label{def:eq}{\em
Given an $\calL$-structure $\calM$ and $a,b\in M$, define $a\sim b$ if and only if for every atomic formula $R(x_1,\ldots, x_s)$ and $m_2,\ldots, m_s\in M\setminus \{a,b\}$, 
$$
\calM\models \Big(R(a,b,m_3,\ldots, m_s)\leftrightarrow R(b,a,m_3,\ldots, m_s)\Big)\wedge\Big(R(a,m_2,\ldots, m_s)\leftrightarrow R(b,m_2,\ldots, m_s)\Big).
$$
In model theory terms, $a\sim b$ if and only if $\qftp^{\calM}(ab/(M\setminus \{a,b\}))=\qftp^{\calM}(ba/(M\setminus \{a,b\}))$.  }
\end{Definition}

We observe that $a\sim b$ holds if and only if the map from $M$ to $M$ fixing $M\setminus \{a,b\}$ and permuting $a$ and $b$ is an automorphism of $\calM$.

\begin{Example}
Suppose $\calM=(M,E)$ is a directed $r$-uniform hypergraph with vertex set $M$ and edge set $E\subseteq M^{\underline{r}}$.  Given $a,b\in M$ an $1\leq i<j\leq r$, let 
\begin{align*}
N_{ij}(a,b)&= \{(c_1,\ldots, c_{r-2})\in M^{r-2}: (c_1,\ldots, c_{i-1},a,c_i,\ldots, c_j, b,c_{j+1},\ldots, c_{r-2})\in E\}\text{ and }\\
N_i(a,b)&=\{(c_1,\ldots,c_{r-1})\in (M\setminus \{b\})^{r-1}: (c_1,\ldots, c_{i-1},a,c_i,\ldots, c_r)\in E\}.
\end{align*}
  Considering $\calM$ as an $\calL=\{R(x_1,\ldots, x_r)\}$ structure in the usual way yields that for all $a, b\in \calM$, $a\sim b$ holds if and only if for all $1\leq i<j\leq r$, $N_{ij}(a,b)=N_{ij}(b,a)$ and $N_i(a,b)=N_i(b,a)$
\end{Example}

It is easy to check that for any $\calL$-structure $\calM$, $\sim$ is an equivalence relation on $\calM$. 

\begin{Definition}\label{def:basic}{\em
A hereditary $\calL$-property $\calH$ is \emph{basic} if there is $k\in \mathbb{N}$ such that every $\calM\in \calH$ has at most $k$ distinct $\sim$-classes. }
\end{Definition}

The main theorem of this section shows that the speeds of basic properties have the form appearing in case (1) of Theorem \ref{thm:mainthm}.

\begin{Theorem}\label{thm:efthm}
Suppose $\calH$ is a basic hereditary $\calL$-property. Then there is $k\in \mathbb{N}$ and rational polynomials $p_1(x),\ldots, p_k(x)$ such that for sufficiently large $n$, $|\calH_n|=\sum_{i=1}^kp_i(n)i^n$. 
\end{Theorem}
 We will see later that something even stronger holds, namely that $\calH$ is basic \emph{if and only if} its speed has the form appearing in case (1) of Theorem \ref{thm:mainthm} (see Corollary \ref{cor:basic}).    {\bf For the rest of this section, $\calH$ is a fixed non-trivial, basic hereditary $\calL$-property.  } 
 
We end this introductory subsection with a historical note.  The equivalence relation of Definition \ref{def:eq} also makes an appearance in \cite{HM:growth} (see section 2 there).  In that paper, the authors show that for a countably infinite $\calL$-structure $\calM$, several properties are equivalent to $\calM$ having finitely many $\sim$-classes. As a direct consequence, we obtain equivalent formulations of basic properties. Specifically, in the terminology of \cite{HM:growth}, a hereditary $\calL$-property $\calH$ is basic if and only if every countable model of $T_{\calH}$ is \emph{finitely partitioned}, if and only if every countable model of $T_{\calH}$ is \emph{absolutely ubiquitous}.  Observe that if $T_{\calH}$ is basic, then it is $\aleph_0$-categorical.

\subsection{Infinite models as templates}\label{countingESP}

Our proof of Theorem \ref{thm:efthm} can be seen as a generalization of the proof of Theorem 20 in \cite{BBW1}.  One idea used in our proof of Theorem \ref{thm:efthm} is to view countably infinite $\calM\models T_{\calH}$ as ``templates'' for finite elements of $\calH$.  In this subsection we fix notation to make this idea precise, and show that the set of finite structures compatible with a fixed template can be described using first order sentences. The main advantage of this approach is that it allows us to leverage the compactness theorem in the next subsection.  

Fix a countably infinite $\calM\models T_{\calH}$. We make a series of definitions related to $\calM$.  First, by our assumption on $\calH$, $\calM$ has finitely many $\sim$-classes. Fix an enumeration of them, say $A_1,\ldots, A_k$, satisfying $0<|A_1|\leq \ldots \leq |A_k|$, and call this the \emph{canonical decomposition} of $\calM$. It is straightforward to check that for each $\tau(x_1,\ldots, x_s)\in \Delta_{neq}$, there is $\Sigma^{\calM}_{\tau}\subseteq [k]^s$  such that
\begin{align*}
\tau^{\calM}=\bigcup\{(A_{i_1}\times \ldots \times A_{i_s})\cap M^{\underline{s}}:(i_1,\ldots, i_s)\in \Sigma^{\calM}_{\tau}\}.
\end{align*}
Let $t=\max\{i\in [k]: A_i \text{ is finite}\}$ and set $K=\max\{r, |A_t|\}$. Given any set $X$, let $\Omega^{\calM}(X)$ denote the set of ordered partitions $(X_1,\ldots, X_k)$ of $X$ satisfying $|X_i|=|A_i|$ for each $i\in [t]$, $\min\{|X_i|:t<i\leq k\}>K$.   Note $(A_1,\ldots, A_k)\in \Omega^{\calM}(M)$.

\begin{Definition}{\em
An $\calL$-structure $\calN$ is \emph{compatible} with $\calM$ if there is $(B_1,\ldots, B_k)\in \Omega^{\calM}(N)$  such that for each $\tau(x_1,\ldots, x_s)\in \Delta_{neq}$, $\tau^{\calN}=\bigcup \{(B_{i_1}\times \ldots \times B_{i_s})\cap N^{\underline{s}}:(i_1,\ldots, i_s)\in \Sigma^{\calM}_{\tau}\}$}.
\end{Definition}

Observe that if $\calN$ is compatible with $\calM$, witnessed by $(B_1,\ldots, B_k)\in \Omega^{\calM}(N)$, then $\{B_1,\ldots, B_k\}$ are the $\sim$-classes of $\calN$.  It is straightforward to see that if $\calN$ is finite and compatible with $\calM$, then $\calN$ is isomorphic to a substructure of $\calM$, and is thus in $\calH$.  For this reason we think of $\calM$ as forming a ``template'' for the finite structures $\calN$ which are compatible with $\calM$.  We now show that being compatible with $\calM$ can be defined using a first-order sentence.  We leave it to the reader to check that there is a formula $\phi(x,y)$ (with quantifiers) such that for any $\calL$-structure $\calG$ and $a,b\in G$, $a\sim b$ if and only if $\calG\models \phi(a,b)$.  We will abuse notation and write $x\sim y$ for this formula. 

\begin{Lemma}
There is a sentence $\theta_{\calM}$ such that for any $\calL$-structure $\calN$, $\calN\models \theta_{\calM}$ if and only if $\calN$ is compatible with $\calM$.
\end{Lemma}
\begin{proof}
Given $\tau(x_1,\ldots, x_s)\in \Delta_{neq}$, let $\phi_{\tau,\calM}(z_1,\ldots, z_k)$ be the following formula.
\begin{align*}
\forall x_1\ldots \forall x_s\Big(\tau(x_1,\ldots, x_s) \leftrightarrow \Big(\Big(\bigwedge_{1\leq i\neq j\leq s}x_i\neq x_j\Big)\wedge\Big( \bigvee_{(i_1,\ldots, i_s)\in \Sigma^{\calM}_{\tau}}\Big(\bigwedge_{j=1}^sx_j\sim z_{i_j} \Big)\Big)\Big).
\end{align*}
Note that for any $(a_1,\ldots, a_k)\in A_1\times \ldots \times A_k$ and $\tau\in \Delta_{neq}$, $\calM\models \phi_{\tau,\calM}(a_1,\ldots, a_k)$.  Define $\theta_{\calM}$ to be the following $\calL$-sentence, where $n_i=|A_i|$ for each $i\in [t]$. 
\begin{align*}
&\exists z_1\ldots \exists z_k\Big(  \Big(\bigwedge_{i\in [t]} \exists^{=n_i}x (x\sim z_i) \Big)\wedge \Big(\bigwedge_{i\in [k]\setminus [t]}\exists^{>K}x(x\sim z_i)\Big) \wedge \Big(\bigwedge_{\tau\in \Delta_{neq}} \phi_{\tau,\calM}(z_1,\ldots, z_k)\Big)\Big).
\end{align*}
We leave it to the reader to verify that for any $\calL$-structure $\calN$, $\calN\models\theta_{\calM}$ if and only if $\calN$ is compatible with $\calM$.
\end{proof}

Observe that for any $\calN\models \theta_{\calM}$, there is a sufficiently saturated elementary extension $\calM\prec \calM'$ such that $\calN$ is isomorphic to a substructure of $\calM'$. Thus $\calN\models Th_{\forall}(\calM)$, so $\calN\models T_{\calH}$, and consequently $age(\calN)\subseteq \calH$.

\subsection{Proof of Theorem \ref{thm:efthm}}

In this subsection we prove Theorem \ref{thm:efthm} by showing the speed of $\calH$ is asymptotically equal to a sum of the form $\sum_{i=1}^kp_i(n)i^n$ for some rational polynomials $p_1,\ldots, p_k$.  Our strategy is as follows. First, we compute the the number of $\calG\in \calH_n$ compatible with a single fixed $\calM\models T_{\calH}$ (Proposition \ref{lem:polcount1}).  We then use the compactness theorem to show there are finitely many $\calM\models T_{\calH}$ which serve as templates for all sufficiently large elements of $\calH$.  This will then allow us to compute the speed of $\calH$.

The first goal of the section is to prove Proposition \ref{lem:polcount1}, which shows the number of elements of $\calH_n$ which are compatible with a fixed $\calM\models T_{\calH}$ is equal to $C|\Omega^{\calM}([n])|$ for some constant $C$ depending on $\calM$. We give a brief outline of the argument here. Given a fixed $\calM\models T_{\calH}$ and $n\in \mathbb{N}$, every element of $\calH_n$ which is compatible with $\calM$ can be constructed by choosing an element of $\calP\in \Omega^{\calM}([n])$, then choosing the realizations of each $\tau\in \Delta_{neq}$ as prescribed by the set $\Sigma_{\tau}^{\calM}$.  This gives an upper bound of $|\Omega^{\calM}([n])|$.  The constant factor then arises from considering double counting.  Dealing with the double counting is the motivation for the next definition.

\begin{Definition}{\em
Suppose $\calM\models T_{\calH}$ is countably infinite and $A_1,\ldots, A_k$ is its canonical decomposition.  Define $Aut^*(\calM)$ to be the set of permutations $\sigma:[k]\rightarrow [k]$ with the property that there is an automorphism $f$ of $\calM$ satisfying $f(A_i)=A_{\sigma(i)}$ for each $i\in [k]$.}
\end{Definition}

We will show in Proposition \ref{lem:polcount1} that the number of element of $\calH_n$ compatible with a fixed $\calM$ is $|\Omega^{\calM}([n])|/|Aut^*(\calM)|$.  We need the next two lemmas for this.

\begin{Lemma}\label{lem:part}
Let $k\in \mathbb{N}^{>0}$, and let $\sigma:[k]\rightarrow [k]$ be a permutation.  Assume $\calM_1,\calM_2\in \calH$ are countably infinite, both have $k$ distinct $\sim$-classes, and $\calM_2$ has at least as many finite $\sim$-classes as $\calM_1$.  If $\sigma(\Omega^{\calM_2}(X))\cap \Omega^{\calM_1}(X)\neq \emptyset$ for some set $X$, then $\sigma(\Omega^{\calM_2}(Y))\subseteq \Omega^{\calM_1}(Y)$ for all sets $Y$.
\end{Lemma}
\begin{proof}
Let $n_1\leq \ldots \leq n_{t_1}$ and $m_1\leq \ldots \leq m_{t_2}$ be the sizes of the finite $\sim$-classes of $\calM_1$ and $\calM_2$, respectively.  Note by assumption, $t_1\leq t_2$.  Set $K_1=\max\{r, n_{t_1}\}$.  By assumption, there is $(X_1,\ldots, X_k)\in \Omega^{\calM_2}(X)$ such that $(X_{\sigma(1)},\ldots, X_{\sigma(k)})\in \Omega^{\calM_1}(X)$.  By definition, we must have that for each $i\in [t_1]$, $n_i=m_{\sigma(i)}$, and for all $i>t_1$, $m_{\sigma(i)}>K_1$. Since $t_1\leq t_2$, this implies $\sigma([t_1])=[t_1]$, and for all $i\in [t_1]$, $n_i=m_i$. Consequently, for all $j>t_1$, $m_j>K_1$. Clearly this implies that for any set $Y$, if $(Y_1,\ldots, Y_k)\in \Omega^{\calM_2}(Y)$, then $(Y_{\sigma(1)},\ldots, Y_{\sigma(k)})\in \Omega^{\calM_1}(Y)$, i.e. $\sigma(\Omega^{\calM_2}(Y))\subseteq \Omega^{\calM_1}(Y)$. 

\end{proof}

The next Lemma gives us useful information about elements of $Aut^*(\calM)$. 

\begin{Lemma}\label{rem:aut}
Suppose $\calM\models T_{\calH}$ is countably infinite, $A_1,\ldots, A_k$ is its canonical decomposition, and $\sigma:[k]\rightarrow [k]$ is a permutation. Then $\sigma\in Aut^*(\calM)$ if and only if for all $\tau(x_1,\ldots, x_s)\in \Delta_{neq}$, $\Sigma^{\calM}_{\tau}$ and $\Omega^{\calM}(M)$ are invariant under $\sigma$. 
\end{Lemma}
\begin{proof}
Let $t=\max\{i\in [k]: A_i \text{ is finite}\}$.  First, suppose $\sigma\in Aut^*(\calM)$.  Then there is an automorphism $f$ of $\calM$ such that for each $i\in [k]$, $f(A_i)=A_{\sigma(i)}$.  This implies that for each $\tau(x_1,\ldots, x_s)\in \Delta_{neq}$,
$$
\tau^{\calM}=\bigcup_{(i_1,\ldots, i_s)\in \Sigma^{\calM}_{\tau}}(A_{i_1}\times \ldots \times A_{i_s})\cap M^{\underline{s}}=\bigcup_{(i_1,\ldots, i_s)\in \Sigma^{\calM}_{\tau}}(A_{\sigma(i_1)}\times \ldots \times A_{\sigma(i_s)})\cap M^{\underline{s}},
$$
which implies $\sigma(\Sigma^{\calM}_{\tau})=\Sigma^{\calM}_{\tau}$, i.e. $\Sigma_{\tau}^{\calM}$ is invariant under $\sigma$.   Further, since $f$ is a bijection, $n_i=n_{\sigma(i)}$ for each $i\in [t]$, and $\sigma([t])=[t]$.  Therefore $\Omega^{\calM}(M)$ is invariant under $\sigma$. 

Suppose conversely that for all $\tau(x_1,\ldots, x_s)\in \Delta_{neq}$, $\Omega^{\calM}(M)$ and $\Sigma^{\calM}_{\tau}$ are invariant under $\sigma$.  Since $\Omega^{\calM}(M)$ is invariant under $\sigma$, $(A_{\sigma(1)},\ldots, A_{\sigma(k)})\in \Omega^{\calM}(M)$.  Therefore, for each $i\in [t]$, $|A_i|=|A_{\sigma(i)}|$, and for each $i\in [k]\setminus [t]$, $|A_i|=|A_{\sigma(i)}|=\aleph_0$.  Thus there is a bijection $f:M\rightarrow M$ satisfying $f(A_i)=A_{\sigma(i)}$ for each $i\in [k]$.  

Now fix $\tau(x_1,\ldots, x_s)\in \Delta_{neq}$ and $(a_1,\ldots, a_s)\in M^{\underline{s}}$.  Suppose $\calM\models \tau(a_1,\ldots, a_s)$. Then $(a_1,\ldots, a_s)\in A_{i_1}\times \ldots \times A_{i_s}$ for some $(i_1,\ldots, i_s)\in \Sigma^{\calM}_{\tau}$.  By definition of $f$, $(f(a_1),\ldots, f(a_s))\in A_{\sigma(i_1)}\times \ldots \times  A_{\sigma(i_s)}$.  Since $\sigma(\Sigma^{\calM}_{\tau})=\Sigma^{\calM}_{\tau}$, $(\sigma(i_1),\ldots, \sigma(i_s))\in \Sigma_{\tau}^{\calM}$, so by definition of $\Sigma_{\tau}^{\calM}$, $\calM\models \tau(f(a_1),\ldots, f(a_s))$.  This shows that $\calM\models \tau(f(a_1),\ldots, f(a_s))$ if and only if $\calM\models \tau(a_1,\ldots, a_s)$ (the ``only if'' part comes from the same argument applied to $f^{-1}$). Thus $f$ is an automorphism of $\calM$ and consequently $\sigma \in Aut^*(\calM)$.   
\end{proof}

\begin{Proposition}\label{lem:polcount1}
For any countably infinite $\calM\models T_{\calH}$ and sufficiently large $n\in \mathbb{N}$, 
$$
|\{\calN\in \calH_n: \calN\models \theta_{\calM}\}|=|\Omega^{\calM}([n])|/|Aut^*(\calM)|.
$$
\end{Proposition}
\begin{proof}
Fix a countably infinite $\calM\models T_{\calH}$  and a large $n\in \mathbb{N}$.  Given $\calP=(X_1,\ldots, X_k)\in \Omega^{\calM}([n])$, let $\calN_{\calP}$ be the structure with domain $[n]$ satisfying, for each $\tau(x_1,\ldots, x_s)\in \Delta_{neq}$, 
$$
\tau^{\calN_{\calP}}=\bigcup\{(X_{i_1}\times \ldots \times X_{i_s})\cap [n]^{\underline{s}}:(i_1,\ldots, i_s)\in \Sigma^{\calM}_{\tau}\}.
$$
By definition, $\calG\in \calH_n$ is compatible with $\calM$ if and only if $\calG=\calN_{\calP}$ for some $\calP\in \Omega^{\calM}([n])$. Thus if we define $\Phi(\calP)=\calN_{\calP}$ for each $\calP\in \Omega^{\calM}([n])$, then $\Phi$ is a function $\Phi:\Omega^{\calM}([n])\rightarrow \calH_n$ satisfying $Im(\Phi)=\{\calN\in \calH_n: \calN\models \theta_{\calM}\}$. It suffices to show that for all $\calG\in Im(\Phi)$, $|\Phi^{-1}(\calG)|=|Aut^*(\calM)|$, since then 
\begin{align*}
|Im(\Phi)|=|\{\calN\in \calH_n: \calN\models \theta_{\calM}\}|=|\Omega^{\calM}([n])|/|Aut^*(\calM)|.
\end{align*}
Fix $\calG\in Im(\Phi)$. By definition, there is a $\calP\in \Omega^{\calM}([n])$ so that $\calG=\calN_{\calP}$.  We show 
\begin{align}\label{align2}
\Phi^{-1}(\calG)=\{\sigma(\calP): \sigma \in Aut^*(\calM)\}.
\end{align}  

Suppose $\sigma\in Aut^*(\calM)$.  By Lemma \ref{rem:aut}, $\Omega^{\calM}(M)\cap \sigma(\Omega^{\calM}(M))\neq \emptyset$, so Lemma \ref{lem:part} implies $\sigma(\Omega^{\calM}([n]))\subseteq \Omega^{\calM}([n])$. Thus $\sigma(\calP)\in \Omega^{\calM}([n])$.  Then, also by Lemma \ref{rem:aut}, $\sigma(\Sigma^{\calM}_{\tau})=\Sigma^{\calM}_{\tau}$ for each $\tau\in \Delta_{neq}$, which implies $\tau^{\calN_{\calP}}={\tau}^{\calN_{\sigma(\calP)}}$.  Thus $\Phi(\sigma(\calP))=\calN_{\sigma(\calP)}=\calN_{\calP}=\calG$, so $\sigma(\calP)\in \Phi^{-1}(\calG)$. 

On the other hand, suppose $\calQ\in \Phi^{-1}(\calG)$.  Note $\calG=\calN_{\calP}=\calN_{\calQ}$ implies there is a permutation $\eta:[k]\rightarrow [k]$ such that $\calQ=\eta(\calP)$. Then $\calQ\in \Omega^{\calM}([n])\cap \eta(\Omega^{\calM}([n]))$ and $\calP\in \Omega^{\calM}([n])\cap \eta^{-1}(\Omega^{\calM}([n]))$ imply by Lemma \ref{lem:part} that $\Omega^{\calM}(M)= \eta(\Omega^{\calM}(M))$.  Further, $\calN_{\calP}=\calN_{\calQ}$ implies that for each $\tau(x_1,\ldots, x_s)\in \Delta_{neq}$, $\eta(\Sigma^{\calM}_{\tau})=\Sigma^{\calM}_{\tau}$.  Thus by Lemma \ref{rem:aut}, $\eta \in Aut^*(\calM)$. Since $Aut^*(\calM)$ is clearly closed under inverses and $\calQ=\eta^{-1}(\calP)$, $\calQ\in \{\sigma(\calP): \sigma \in Aut^*(\calM)\}$. Thus we have shown (\ref{align2}) and $|\Phi^{-1}(\calG)|=|Aut^*(\calM)|$. 
\end{proof}

Lemma \ref{lem:polcount0} below, proved in \cite{BBW1}, will be used to compute $|\Omega^{\calM}([n])|$ for a fixed $\calM\models T_{\calH}$. 

\begin{Lemma}[Lemma 19 in \cite{BBW1}]\label{lem:polcount0}
Suppose $\ell,s,c_1,\ldots, c_{t}$ are integers.  If $c_1,\ldots, c_t\leq s$, then there are rational polynomials $p_1,\ldots, p_{\ell}$ such that the following holds for all $n\geq \ell s+c$, where $c=\sum_{i=1}^tc_i$.
\begin{align}\label{lem19}
\sum_{n_1,\ldots, n_{\ell}>s, \sum_{i=1}^{\ell}n_i=n-c} {n\choose n_1,\ldots, n_{\ell},c_1,\ldots, c_t}=\sum_{i=1}^{\ell}p_i(n)i^n.
\end{align}
Further, if $\ell=1$, then $p_1$ has degree $c$.
\end{Lemma}

Note that for any $\calM\models T_{\calH}$, $|\Omega^{\calM}([n])|$ is by definition of the form appearing in the left hand side of Lemma \ref{lem:polcount0}.  We use this along with Proposition \ref{lem:polcount1} to compute the number of $\calG\in \calH_n$ compatible with a fixed $\calM\models T_{\calH}$.

\begin{Corollary}\label{cor:polcor}
Suppose $\calM\models T_{\calH}$ is countably infinite with $\ell$ infinite $\sim$-classes.  Then there are rational polynomials $p_1,\ldots, p_{\ell}$ such that for all sufficiently large $n\in \mathbb{N}$, 
$$
|\{\calN\in \calH_n: \calN\models \theta_{\calM}\}|=\sum_{i=1}^{\ell}p_i(n)i^n.
$$
Further, when $\ell=1$, the degree of $p_1(x)$ is equal to the number of elements of $\calM$ in a finite $\sim$-class.
\end{Corollary}
\begin{proof}
Let $c_1\leq \ldots\leq c_t$ be the sizes of the finite $\sim$-classes of $\calM$.  Set $c=\sum_{i=1}^tc_i$, and $K=\max\{r,c_t\}$.  By definition, if $\ell=k-t$, then
\begin{align}\label{align11}
|\Omega^{\calM}([n])|=\sum_{n_1,\ldots, n_{\ell}>K, \sum_{i=1}^{\ell}n_i=n-c} {n\choose n_1,\ldots, n_{\ell},c_1,\ldots, c_t}.
\end{align}
By Lemma \ref{lem:polcount0}, there are rational polynomials $q_1,\ldots, q_{\ell}$ such that for large enough $n$, $|\Omega^{\calM}([n])|=\sum_{i=1}^{\ell}q_i(n)i^n$.  Further, if $\ell=1$, then $q_1(x)$ has degree $c$. Combining this with Proposition \ref{lem:polcount1}, we obtain that for sufficiently large $n$, 
$$
|\{\calN\in \calH_n: \calN\models \theta_{\calM}\}|=|\Omega^{\calM}([n])|/|Aut^*(\calM)|=\Big(\sum_{i=1}^{\ell}q_i(n)i^n\Big)/|Aut^*(\calM)|=\sum_{i=1}^{\ell}p_i(n)i^n,
$$
where each $p_i(x)$ is the rational polynomial obtained by dividing the coefficients of $q_i(x)$ by the integer $|Aut^*(\calM)|$.
\end{proof}

We now prove our final lemma, which reduces the problem of counting the number of $\calG\in \calH_n$ compatible with finitely many templates to the problem of counting the number compatible with a single template.

\begin{Lemma}\label{lem:polcount2}
Suppose $\ell\geq 1$ is an integer, $\calM_1,\ldots, \calM_{\ell}\models T_{\calH}$ are countably infinite, and $\theta_{\calM_1}\wedge \ldots \wedge \theta_{\calM_{\ell}}$ is satisfiable.  Then there is $i\in [\ell]$ such that $ \theta_{\calM_1}\wedge \ldots \wedge \theta_{\calM_{\ell}}\equiv \theta_{\calM_i}$.
\end{Lemma}
\begin{proof}
By induction it suffices to do the proof for $\ell=2$.
Suppose $\calM_1, \calM_2\models T_{\calH}$ are countably infinite and $\theta_{\calM_1}\wedge \theta_{\calM_2}$ is satisfiable.  Clearly this implies $\calM_1$ and $\calM_2$ must have the same number of $\sim$-classes, say this number is $k\geq 1$.  Without loss of generality, assume $\calM_2$ has at least as many finite $\sim$-classes as $\calM_1$. 

By assumption, there is an $\calL$-structure $\calB$ satisfying both $\theta_{\calM_1}$ and $\theta_{\calM_2}$.  Since $\calB\models \theta_{\calM_2}$, there is $(B_1,\ldots, B_k)\in \Omega^{\calM_2}(B)$ so that for each $\tau(x_1,\ldots, x_s)\in \Delta_{neq}$,
 \begin{align}\label{11}
 \tau^{\calB}=\bigcup\{(B_{j_1}\times \ldots \times B_{j_s})\cap B^{\underline{s}}:(j_1,\ldots, j_s)\in \Sigma_{\tau}^{\calM_2}\}.
 \end{align}
Since $\calB\models \theta_{\calM_1}$, there is a permutation $\sigma:[k]\rightarrow [k]$ so that $(B_{\sigma(1)},\ldots, B_{\sigma(k)})\in \Omega^{\calM_1}(B)$ and for all $\tau(x_1,\ldots, x_s)\in \Delta_{neq}$,
 \begin{align}\label{2}
 \tau^{\calB}=\bigcup\{(B_{\sigma(j_1)}\times \ldots \times B_{\sigma(j_s)})\cap B^{\underline{s}}:(j_1,\ldots, j_s)\in \Sigma_{\tau}^{\calM_1}\}.
 \end{align}
Observe (\ref{11}) and (\ref{2}) imply $\sigma(\Sigma_{\tau}^{\calM_1})=\Sigma_{\tau}^{\calM_2}$ for all $\tau\in \Delta_{neq}$.  Further, $\sigma(\Omega^{\calM_2}(B))\cap \Omega^{\calM_1}(B)\neq \emptyset$, so Lemma \ref{lem:part} implies that $\sigma(\Omega^{\calM_2}(Y))\subseteq \Omega^{\calM_1}(Y)$ for any set $Y$.  

We now show $\theta_{\calM_1}\wedge \theta_{\calM_2}\equiv \theta_{\calM_2}$.  Clearly it suffices to show $\theta_{\calM_2}\models \theta_{\calM_1}$.  Fix $\calN\models \theta_{\calM_2}$.  Then there is $(N_1,\ldots, N_k)\in \Omega^{\calM_2}(N)$ so that for each $\tau(x_1,\ldots, x_s)\in \Delta_{neq}$ 
\begin{align}\label{5}
 \tau^{\calN}=\bigcup \{(N_{i_1}\times \ldots \times N_{i_s})\cap N^{\underline{s}}: (i_1,\ldots, i_s)\in \Sigma_{\tau}^{\calM_2}\}.
\end{align}
Since $\sigma(\Omega^{\calM_2}(N))\subseteq \Omega^{\calM_1}(N)$, we have $(N_{\sigma(1)},\ldots, N_{\sigma(k)})\in \Omega^{\calM_1}(N)$.  Further, for each $\tau(x_1,\ldots, x_s)\in \Delta_{neq}$, $\sigma(\Sigma_{\tau}^{\calM_1})=\Sigma_{\tau}^{\calM_2}$, so  (\ref{5}) implies that 
\begin{align*}
\tau^{\calN}&=\bigcup\{(N_{\sigma(i_1)}\times \ldots \times N_{\sigma(i_s)})\cap N^{\underline{s}}:(i_1,\ldots, i_s)\in \Sigma_{\tau}^{\calM_1}\}.
\end{align*}
Thus by definition, $\calN\models \theta_{\calM_1}$.
\end{proof}

We now prove the main result of this subsection. The proof uses the compactness theorem and the preceding lemma to show the speed of $\calH$ is a linear combination of finitely many functions of the form appearing in Corollary \ref{cor:polcor}.

\vspace{5mm}

\begin{Theorem}\label{thm:efthm1}
There are $k\in \mathbb{N}$ and rational polynomials $p_1(x),\ldots, p_k(x)$ such that for sufficiently large $n$, $|\calH_n|=\sum_{i=1}^kp_i(n)i^n$. 
\end{Theorem}
\begin{proof}
Clearly the following set of sentences is inconsistent.
$$
T_{\calH}\cup \{\neg \theta_{\calM}: \calM\models T_{\calH}\text{ is countably infinite}\}\cup \{\exists x_1 \ldots \exists x_n\bigwedge_{i\neq j}x_i\neq x_j: n\geq 1\}.
$$
Thus by compactness, there are finitely many $\theta_{\calM_1},\ldots, \theta_{\calM_k}$ such that for sufficiently large $n$, any element of $\calH_n$ must satisfy $\bigvee_{i=1}^k \theta_{\calM_i}$.  Combining this with the inclusion/exclusion principle yields that for large $n$,
\begin{align*}
|\calH_n|&=|\{\calM\in \calH_n: \calM\models \theta_{\calM_1}\vee \ldots \vee \theta_{\calM_k}\}|\\
&=\sum_{u=1}^k(-1)^{u+1}\Big(\sum_{1\leq i_1<\ldots <i_{u}\leq k}|\{\calM\in \calH_n: \calM\models \theta_{\calM_{i_1}}\wedge \ldots \wedge \theta_{\calM_{i_u}}\}|\Big).
\end{align*}
Apply Corollary \ref{cor:polcor} and Lemma \ref{lem:polcount2} to finish the proof.
\end{proof}

Theorem \ref{thm:efthm1} proves Theorem \ref{thm:efthm}, since $\calH$ was an arbitrary basic non-trivial hereditary $\calL$-property.  We have actually shown more about basic properties, which we sum up in Corollary \ref{cor:polcount1} below.  We leave the proof to the reader, as it follows from the proof of Theorem \ref{thm:efthm1} and the fact that for any $\calM\models T_{\calH}$, $\{\calN\in \calH: \calN\models \theta_{\calM}\}\subseteq age(\calM)\subseteq \calH$.
\begin{Corollary}\label{cor:polcount1}
Suppose $\calH$ is a non-trivial basic hereditary $\calL$-property.  Then there is a trivial hereditary $\calL$-property $\calF$ and finitely many countably infinite basic $\calL$-structures $\calM_1,\ldots, \calM_m$ such that $\calH=\calF\cup \bigcup_{i=1}^mage(\calM_i)$.  

Moreover the following hold, where for each $1\leq i\leq m$, $m_i$ is the number of elements of $\calM_i$ in a finite $\sim$-class and $\ell_i$ is the number of infinite $\sim$-classes of $\calM_i$.
\begin{enumerate}
\item If $\ell=\max\{\ell_i:i\in [m]\}=1$ then for large $n$,  $|\calH_n|=p(n)$ where $p(n)$ is a rational polynomial of degree $c:=\max\{m_i: i\in [m]\}$.
\item If $\ell =\max\{\ell_i: i\in [m]\}\geq 2$, then for large $n$,  $|\calH_n|=\sum_{i=1}^{\ell}p_i(n)i^n$ where each $p_i(n)$ is a rational polynomial.
\end{enumerate}
\end{Corollary}

The structural dichotomy between cases (1) and (2) in Corollary \ref{cor:polcount1} also made an appearance in  \cite{HM:growth}. Specifically, in the terminology of \cite{HM:growth}, case (1) in Corollary \ref{cor:polcount1} holds if and only if every model $\calM\models T_{\calH}$ is absolutely $|M|$-ubiquitous.

\section{Totally bounded properties}\label{sec:tb}
In this section we prove results about a very restricted class of properties, namely those which are \emph{totally bounded}. The main result of this section is Theorem \ref{thm:tb} which tells us about the speeds of totally bounded properties.  Totally bounded hereditary $\calL$-properties behave much like hereditary graph properties with uniformly bounded degree, and our proofs in this section largely follow the corresponding proofs for these kinds of graphs (see Lemmas 24 and 25 in \cite{BBW1}). The results of this subsection will be used in Section \ref{sec:ma} to prove counting dichotomies for more general classes.

In this section $\calL$ is a finite relational language of arity $r\geq 0$.  Throughout, $C$ denotes the set of constants of $\calL$.

\begin{Definition}\label{def:tb}{\em 
An $\calL$-structure $\calM$ is \emph{totally $k$-bounded} if for every relation symbol $R(x_1,\ldots, x_s)$ in $\calL$ and every partition $[s]=I\cup J$ into nonempty sets $I$ and $J$, 
$$
\calM\models \forall \xbar_I\exists^{<k}\xbar_JR(x_1,\ldots, x_s).
$$
A hereditary $\calL$-property $\calH$ is \emph{totally bounded} if there is an integer $k$ such that every $\calM\in \calH$ is totally $k$-bounded.
}
\end{Definition}

For example, if $k\in \mathbb{N}$ and $\calH$ is a hereditary graph property, then $\calH$ is totally $k$-bounded if and only if all of the graphs in $\calH$ have maximum degree less than $k$.  Note that an $\calL$-structure $\calM$ is totally $k$-bounded if and only if for every relation $R(x_1,\ldots, x_s)$ in $\calL$, and every partition $[s]=I\cup J$ \emph{with $|I|=1$}, $\calM\models \forall \xbar_I\exists^{<k}\xbar_JR(x_1,\ldots, x_s)$.

We begin by considering a generalization of the notion of a connected component in a graph. Given an $\calL$-structure $\calM$ and $a,b\in M$, a \emph{path from $a$ to $b$} is a finite sequence, $\abar_1,\ldots, \abar_k$, of tuples of elements of $M$ such that the following hold.
\begin{enumerate}
\item For each $1\leq i\leq k-1$, $(\cup \abar_i)\cap (\cup \abar_{i+1})\neq \emptyset$.
\item $a\in \cup \abar_1$ and $b\in \cup \abar_k$.
\item $\calM\models \psi_{1}(\abar_1)\wedge \ldots \wedge \psi_{k}(\abar_k)$ for some relations $\psi_{1}(\xbar_{1}),\ldots, \psi_{k}(\xbar_{k})$ from $\calL$.
\end{enumerate}
The \emph{length} of the path is $k$.  We say a subset $A\subseteq M$ is \emph{connected} if for all $a\neq b\in A$, there is a path from $a$ to $b$ which is contained in $A$. When $\calM$ is a graph, then these are just the usual graph theoretic notions of a path and of a connected set.  

\begin{Definition}{\em
Suppose $\calM$ is an $\calL$-structure, and $A\subseteq M$.  We say that $A$ is a \emph{component in $\calM$} if it is a maximal connected set, i.e. if $A$ is connected and for all $a\in A$ and $c\in M$, if there is a path from $a$ to $c$, then $c\in A$.}
\end{Definition}

We would like to point out that the notion of components and $\sim$-classes are different. For example, if $\calM$ is an infinite graph with no edges, then $\calM$ has infinitely many components (since each vertex is in a component of size $1$), but only one $\sim$-class.  On the other hand, if $\calM$ is an infinite path, then it it has only one component, but infinitely many $\sim$-classes.

Given a hereditary $\calL$-property $\calH$ and $m\in \mathbb{N}^{>0}$, we say $\calH$ has \emph{infinitely many components of size $m$} if there is $\calM\models T_{\calH}$ such that $\calM$ has infinitely many distinct components of size $m$. We say $\calH$ has \emph{infinitely many components} if it there is $\calM\models T_{\calH}$ with infinitely many components.  Otherwise we say $\calH$ has finitely many components. We say $\calH$ has \emph{finite components} if there is $K\in \mathbb{N}$ such that for every $\calM\models T_{\calH}$, every component of $\calM$ has size at most $K$.  Otherwise $\calH$ has \emph{infinite components}. Note that if $\calH$ is non-trivial and has finite components, then it must have infinitely many components.

Our first goal is to prove two lemmas about hereditary $\calL$-properties with restrictions placed on their components. Specifically, Lemma \ref{lem:lowerboundcomp} will give us lower bounds for any hereditary $\calL$-property with infinitely many components of a fixed finite size, and Lemma \ref{lem:exact} will characterize the speeds of hereditary $\calL$-properties with finite components.  We require the following lemma, which which appears within the proof of Lemma 24 in \cite{BBW1}.

\begin{Lemma}\label{lem:lb}
Suppose $k, n\in \mathbb{N}^{>0}$ and $k\ll n$. Then the number of ways to partition $k\lfloor n/k\rfloor$ into $\lfloor n/k\rfloor$ parts of size $k$ is at least $n^{n(1-1/k-o(1))}$.
\end{Lemma}
\begin{proof}
Set $m=k\lfloor n/k\rfloor$ and let $f(n)$ be the number of ways to partition $[m]$ into $\ell:=\lfloor n/k\rfloor$ parts, each of size $k$.  Clearly $f(n)\geq {m\choose k,\ldots, k}\frac{1}{\ell!}$.  Using Stirling's approximation and the definitions of $\ell$ and $m$, we obtain the following.
$$
{m\choose k,\ldots, k}\frac{1}{\ell!}= \frac{m!}{(k!)^{\lfloor \frac{n}{k}\rfloor}\lfloor \frac{n}{k}\rfloor!}\geq \frac{\sqrt{2\pi}m^{m+\frac{1}{2}}e^{-m}}{(k!)^{\lfloor \frac{n}{k}\rfloor}( \frac{n}{k})^{\frac{n}{k}+\frac{1}{2}}e^{1-\frac{n}{k}}}=m^{m(1-o(1))}n^{-\frac{n}{k}(1+o(1))}=n^{n(1-\frac{1}{k}-o(1))},
$$
where the last two equalities are because $n-m<k$ and $n\gg k$. Thus $f(n)\geq n^{n(1-1/k-o(1))}$.
\end{proof}

\begin{Lemma}\label{lem:lowerboundcomp}
Suppose $\calH$ is a hereditary $\calL$-property and $k\geq 2$ is an integer. Assume $\calH$ has infinitely many components of size $k$. Then $|\calH_n|\geq n^{(1-1/k-o(1))n}$. 
\end{Lemma}
\begin{proof}
Suppose $\calH$ has infinitely many components of size $k$. By definition, there is $\calM\models T_{\calH}$ with infinitely many distinct components, each of size $k$. Let $D=C^{\calM}$.  Since $\calL$ is finite, we can find  distinct components $\{A_i: i\in \mathbb{N}\}$, each of which has size $k$ and is disjoint from $D$ (since $D$ is finite, and each element of $D$ is in at most one component).

Fix $n\gg |D|$ and set $\ell=\lfloor (n-|D|)/k\rfloor$.  Now choose $B\subseteq A_{\ell+1}$ of size $n-|D|-k\ell$ and set $A=A_1\sqcup \ldots \sqcup A_{\ell}\sqcup B\sqcup D$.  Observe $|B|<k$, $|A|=n$ and for every bijection $f:A\rightarrow [n]$, $f(\calM[A])\in \calH_n$.  Note that the only components of size $k$ in $f(\calM[A])$ which are also disjoint from $f(D)$ are $f(A_1),\ldots, f(A_{\ell})$.

Fix $f_0:D\cup B\rightarrow [n]$.  Suppose $f$ and $f'$ are bijections from $A$ to $[n]$ extending $f_0$ with $\{f(A_1),\ldots, f(A_{\ell})\}\neq \{f'(A_1),\ldots, f'(A_{\ell})\}$. Then clearly $f(\calM[A])\neq f'(\calM[A])$, since these structures disagree on what are the components of size $k$ disjoint from $f_0(D)$.  Therefore $|\calH_n|$ is at least the number of distinct ways to choose $\ell$ disjoint sets of size $k$ in $[n_0]$ where $n_0=n-|B|-|D|$.  By Lemma \ref{lem:lb}, this is at least $n_0^{n_0(1-1/k-o(1))}$.  Since $|B|<k$, $|D|$ is constant, and $n$ is large, this shows $|\calH_n|\geq n^{(1-1/k-o(1))n}$.  
\end{proof}

We now consider the case where $\calH$ has finite components.  We will use the following fact, which is a consequence of the inequality of arithmetic and geometric means.

\begin{Fact}\label{factorial}
Suppose $a_1,\ldots, a_t\in \mathbb{N}^{>0}$. Then $a_1!\ldots a_t!\geq (((\sum_{i=1}^t a_i)/t)!)^t$.
\end{Fact}

\begin{Lemma}\label{lem:exact}
Suppose $\calH$ has finite components and $k$ is the largest integer such that $\calH$ contains infinitely many components of size $k$. If $k=1$ then $\calH$ is basic.  If $k\geq 2$ then $|\calH_n|= n^{(1-1/k-o(1))n}$.
\end{Lemma}

\begin{proof}
Suppose first that $k=1$.  Then there is a fixed integer $w$ such that for any $\calM\models T_{\calH}$, all but $w$ elements of $M$ are in a component of size $1$.  Fix $\calM\models T_{\calH}$. Let $D=C^{\calM}$, and let $X\subseteq M$ be the set of elements contained in a component of size greater than $1$.  Observe that for all $a\neq b\in M\setminus (X\cup D)$, $a\sim b$ if and only if for every relation $R(x_1,\ldots, x_s)$ of $\calL$, $\calM\models R(a,\ldots, a)\leftrightarrow R(b,\ldots, b)$.  Consequently, the number of distinct $\sim$-classes of $\calM$ is at most $|X|+|D|+2^{|\calL|}\leq w+|C|+2^{|\calL|}$.  Since this bound does not depend on $\calM$, this shows $\calH$ is basic.

Suppose now $k\geq 2$.  That $|\calH_n|\geq n^{(1-1/k-o(1))n}$ is immediate from Lemma \ref{lem:lowerboundcomp}.  We now show that $|\calH_n|\leq n^{(1-1/k+o(1))n}$.  By choice of $k$, there is an integer $w$ such that for all $\calM\models T_{\calH}$, all but $w$ elements of $M$ are contained in a component of size at most $k$ in $\calM$.  Let $c=|C|$, and let $d=\sum_{i=1}^k |\calH_i|$. By convention, let $\calH_0$ consist of the empty structure, so $|\calH_0|=1$.  Suppose now $n$ is large.  Then we can construct every $\calG\in \calH_n$ as follows.
\begin{enumerate}[1.]
\item Choose  $D=C^{\calG}$, the interpretations of the constants in $\calG$. There are at most $n^c$ ways to do this.

\item Choose a set $A\subseteq [n]$ of size at most $w$ and choose $\calG[A]$.  The number of ways to do this is at most $\sum_{i=0}^w{n\choose i}|\calH_i|\leq (w+1)n^w2^{|\calL|w^r}$.

\item Choose a sequence of natural numbers $(b_1,\ldots, b_n)$ (some of which may be $0$) such that each $b_i\leq k$ and $\sum_{i=1}^nb_i=|[n]\setminus A|$.  Then partition $[n]\setminus A$ into parts $B_1,\ldots, B_n$ of sizes $b_1,\ldots, b_n$, respectively (note some of the $B_i$ may be empty). The number of ways to do this step is at most $\sum_{\{(b_1,\ldots, b_n)\in [n]^n: b_i\leq k, \sum b_i=n-|A|\}}{n\choose b_1,\ldots, b_n}\leq k^nn^n$.  

\item  Choose a sequence $(\calG_1,\ldots, \calG_n)$ such that for each $1\leq i\leq n$, $\calG_i\in \calH_{b_i}$.  Then make $\calG[B_i]$ isomorphic to $\calG_i$ via the order-preserving bijection from $[b_i]$ to $B_i$. There are at most $(d+1)^n$ ways to do this.  

\item For all relations $R(\xbar)$ in $\calL$ and $\abar\in [n]^{|\xbar|}\setminus ((A^{|\xbar|}\cup \bigcup_{i=1}^n B_i^{|\xbar|})$, let $\calG\models \neg R(\abar)$.  There is only one way to do this given our previous choices.
\end{enumerate}
This yields the following upper bound (recall $c,w, |\calL|, d$ are all constants).
\begin{align}\label{align9}
|\calH_n|&\leq n^c(w+1)n^w2^{|\calL|w^r} k^n n^n (d+1)^n=n^{n(1+o(1))}.
\end{align}
We now consider how many times each $\calG\in \calH_n$ was counted. Fix $\calG\in \calH_n$, and assume $\calG$ is constructed from the sets $D=C^{\calG}$ (step 1), $A$ (step 2), and the sequences $(b_1,\ldots, b_n)$, $(B_1,\ldots, B_n)$, and $(\calG_1,\ldots, \calG_n)$ (steps 3, 4 and 5 respectively).  Let $\calN_1,\ldots, \calN_d$ enumerate all the distinct elements of $\cup_{i=1}^k\calH_i$, and for each $i\in [d]$, let $J_i=\{j\in [n]: \calG_j=\calN_i\}$ and set $a_i=|J_i|$.  Then for any permutation $\sigma:[n]\rightarrow [n]$ satisfying $\sigma(J_i)=J_i$ for each $i\in [d]$, $\calG$ is also generated by making the same choices in steps 1 and 2, while in steps 3-5 choosing the sequences $(b_1,\ldots, b_n)$, $(B_{\sigma(1)},\ldots, B_{\sigma(n)})$, and $(\calG_1,\ldots, \calG_n)$.  So $\calG$ is counted at least $a_1!\cdots a_d!$ times.  Observe $n-w\leq n-|A|=\sum_{i=1}^d|N_i|a_i\leq k\sum_{i=1}^da_i$.  Combining this inequality with Fact \ref{factorial}, we obtain
\begin{align*}
a_1!\cdots a_d!&\geq \Big(\Big(\Big(\sum_{i=1}^d a_i\Big)/d\Big)!\Big)^d\geq \Big(\Big(\frac{n-w}{kd}\Big)!\Big)^d\geq\Big(\frac{n-w}{kd}\Big)^{\frac{n-w}{k}}\geq n^{n/k-o(n)},
\end{align*}
where the last inequality is because $n$ is large and $w, c, k, d$ are constants.  Therefore each $\calG$ is counted at least $n^{n/k-o(1)}$ many times.  Combining this with (\ref{align9}) yields that $|\calH_n|\leq n^{n(1+o(1))}n^{-n/k+o(n)}=n^{n(1-\frac{1}{k}+o(1))}$.
\end{proof}

Note that by Lemma \ref{lem:exact}, we now understand the speed of hereditary properties with finite components. The next two lemmas will help us understand the case of a totally bounded property with infinite components.  Specifically, they show that given a totally bounded $\calM$, if $\calM$ has an infinite component, then by deleting elements, we can find a substructure of $\calM$ with infinitely many components of size $k$ for arbitrarily large finite $k$. In the proof of Theorem \ref{thm:tb} we will combine this with Lemma \ref{lem:lowerboundcomp} to show that if a totally bounded hereditary $\calL$-property $\calH$ has infinite components, then $|\calH_n|\geq n^{n(1-o(1))}$.  

We will use the following notion of distance in an $\calL$-structure $\calM$.  Given $a, b\in M$, define the \emph{distance from $a$ to $b$ in $\calM$}, $d_{\calM}(a,b)$, to be $0$ if $a=b$, to be $\infty$ if no path exists between $a$ and $b$ in $\calM$, and otherwise to be the minimum length of a path from $a$ to $b$ in $\calM$.  Given $a\in A$, and $i\in \mathbb{N}$, define
$$
B^\calM_i(a):=\{e\in M: d_\calM(a,e)\leq i \}.
$$
Then for $X\subseteq A$ and $i\in \mathbb{N}$, set $B^\calM_i(X)=\bigcup_{a\in X}B^\calM_i(a)$.  Observe that if $\calM'\subseteq \calM$ and $a,b\in M'$, then $d_\calM(a,b)\leq d_{\calM'}(a,b)$.

\begin{Lemma}\label{lem:findcomps1}
Suppose $t\geq 1$ is an integer, and $\calL$ has maximum arity $r \geq 2$.  Assume $\calM$ is an $\calL$-structure which contains an infinite connected set $A$.  Then for any $a\in A$, there is $t\leq t'\leq tr$ and a connected set $A'\subseteq B^\calM_t(a)$ with $|A'|=t'$.
\end{Lemma}
\begin{proof}
Fix $a\in A$.  Since $A$ is connected and infinite, there is a relation $\psi_1(\xbar_1)$, and $\abar_1\in A^{|\xbar_1|}$ such that $\psi_1(\abar_1)$ and $a\in \abar_1$. Suppose $1\leq i<t$ and we have chosen $\abar_1,\ldots, \abar_i$ such that $(\cup \abar_1)\cup \ldots \cup (\cup \abar_i)$ is a connected subset of $A$, and $i\leq |(\cup \abar_1)\cup \ldots \cup (\cup \abar_i)|\leq ir$.  Since $A$ is infinite and connected, there is some $\abar_{i+1}\in A^{|\xbar_{i+1}|}$ and a relation $\psi_{i+1}(\xbar_{i+1})$, such that $M\models \psi_{i+1}(\abar_{i+1})$,  $(\cup \abar_{i+1})\cap ((\cup \abar_1)\cup \ldots \cup (\cup \abar_i))\neq \emptyset$, and $(\cup \abar_{i+1})\setminus ((\cup \abar_1)\cup \ldots \cup (\cup \abar_i))\neq \emptyset$. Note $i+1\leq |(\cup \abar_1)\cup \ldots \cup (\cup \abar_{i+1})|\leq (i+1)r$. After $t$ steps, $A':=(\cup \abar_1)\cup \ldots \cup (\cup \abar_t)$ will be connected with $t\leq |A'|\leq tr$. By construction, $A'\subseteq B^\calM_t(a)$.
\end{proof}

\begin{Lemma}\label{lem:findcomps}
Suppose $k,t\geq 1$ are integers, and $\calL$ has maximum arity $r \geq 2$. Assume $\calM$ is a totally $k$-bounded $\calL$-structure, and $\calM$ contains an infinite component.  Then there is $t\leq t'\leq tr$ and a substructure $\calM'\subseteq_{\calL} \calM$ so that $\calM'$ contains infinitely many distinct components of size $t'$.
\end{Lemma}
\begin{proof}
Let $A$ be an infinite component of $\calM$.  Observe that since $\calM$ is totally $k$-bounded, we have that for any finite $X\subseteq M$ and any $s\in \mathbb{N}$, $B^\calM_s(X)$ is finite.  Let $D=C^{\calM}$. Then $B^\calM_1(D)$ is finite, so $\calM[A\setminus B^\calM_1(D)]$ has finitely many components.  One of them must be infinite, call this $A_0$.  

Since $A_0$ is an infinite component, and since $B^\calM_s(X)$ is finite for all $s\in \mathbb{N}$ and finite $X\subseteq M$, there exists a set $\{a_i: i\in \mathbb{N}\}\subseteq A_0$ such that for each $i\neq j$, $d_\calM(a_i,a_j)>2t+1$.  By Lemma \ref{lem:findcomps1}, we may choose for each $i\in \mathbb{N}$ a connected set $C_i\subseteq B^\calM_t(a_i)$ with $t\leq |C_i|\leq tr$.  Note that by construction, for all $i\in \mathbb{N}$, $B^\calM_1(C_i)\cap (D\cup \bigcup_{j\in \mathbb{N}\setminus \{i\}}C_j)=\emptyset$.  

Let $\calM':=\calM[D\cup \bigcup_{i\in \mathbb{N}}C_i]$.  Fix $i\in \mathbb{N}$. We show $C_i$ is a component of $\calM'$.   Clearly $C_i$ is connected in $\calM'$.  Suppose towards a contradiction there is $c\in C_i$ connected to some $a\in M'\setminus C_i$ by a finite path $\bbar_1,\ldots, \bbar_s$ in $\calM'$.  Then for some $1\leq u\leq s$, we have $\bbar_u\cap C_i\neq \emptyset$ and $\bbar_u\cap (M'\setminus C_i)\neq \emptyset$.  But now $B_1^{\calM'}(C_i)\cap (D\cup \bigcup_{j\in \mathbb{N}\setminus \{j\}}C_j)\neq \emptyset $, which is a contradiction since $B_1^{\calM'}(C_i)\subseteq B_1^{\calM}(C_i)$.  Thus each $C_i$ is a component of $\calM'$.  By the pigeon hole principle, there is some $t\leq t'\leq tr$ such that infinitely many $C_i$ have size $t'$.  This finishes the proof. 
\end{proof}

We can now prove our counting theorem for totally bounded properties.   

\begin{Theorem}\label{thm:tb}
Suppose $\calH$ is totally bounded. Then either $\calH$ is basic, $|\calH_n|\geq n^{n(1-o(1))n}$, or for some integer $k\geq 2$, $|\calH_n|=n^{n(1-1/k-o(1))}$.
\end{Theorem}
\begin{proof}
Clearly if $\calL$ has maximum arity $r\leq 1$, then $\calH$ is basic.  So assume $\calL$ has maximum arity $r\geq 2$.  Suppose $\calH$ has infinite components. Then there is $\calM\models T_{\calH}$ with an infinite component.  Lemma \ref{lem:findcomps} implies that for all $t$, there is $t\leq t'\leq tr$ such that $\calH$ has infinitely many components of size $t'$. By Lemma \ref{lem:lowerboundcomp}, $|\calH_n|\geq n^{n(1-1/tr-o(1))}$ for all $t\geq 1$, consequently $|\calH_n|\geq n^{n(1-o(1))}$. 

Assume now that $\calH$ has finite components. Then there is an integer $m$ such that for every $\calM\models T_{\calH}$, every component of $\calM$ has size at most $m$ and there is a maximal $m'\leq m$ such that $\calH$ contains infinitely many components of size $m'$.  By Lemma \ref{lem:exact}, if $m'=1$ then $\calH$ is basic. Otherwise $m'\geq 2$, and Lemma \ref{lem:exact} implies $|\calH_n|=n^{n(1-1/m'-o(1))}$. 
\end{proof}

Note Theorem \ref{thm:tb} shows that if $\calH$ is totally bounded, then $|\calH_n|\geq n^{n(1-o(1))}$ if and only if $\calH$ has infinite components.  Given $\ell \geq 2$, a hereditary $\calL$-property $\calH$ is \emph{factorial of degree $\ell$} if $|\calH_n|=n^{n(1-1/\ell -o(1))}$.  The proof of Theorem \ref{thm:tb} shows that if $\calH$ is totally bounded and factorial of degree $\ell$, then $\ell$ is the largest integer so that $\calH$ has infinitely many components of size $\ell$.  In fact we can show something stronger.

\begin{Corollary}\label{lem:tbb}
Suppose $\calH$ is a totally bounded hereditary $\calL$-property with finite components and $\ell\geq2$ is an integer. The there are finitely many countably infinite $\calL$-structures $\calM_1,\ldots, \calM_m$, each of which is totally bounded with finite components, and a trivial property $\calF$ such that $\calH=\calF\cup \bigcup_{i=1}^m age(\calM_i)$.  

Further, $\calH$ is factorial of degree $\ell$ if and only if $\ell$ is the largest integer such that for some $i\in [m]$, $\calM_i$ has infinitely many components of size $\ell$.
\end{Corollary}
\begin{proof}
Since $\calH$ has finite components there are integers $t\geq 1$ and $w\geq 0$ such that there exists $\calM\models T_{\calH}$ with infinitely many components of size $t$, but for all $\calM'\models T_{\calH}$, there are at most $w$ elements of $\calM'$ in a component of size strictly larger than $t$.  Since $\calL$ is finite, this implies there are only finitely many non-isomorphic $\calL$-structures with domain $\mathbb{N}$ and satisfying $T_\calH$, say $\calM_1,\ldots, \calM_m$.  Since each $\calM_i$ is totally bounded and has all but $w$ elements in a component of size at most $t$, it is straightforward to see that $Th_{\forall}(\calM_i)$ is finitely axiomatizable, in fact by a single sentence, say $\psi_i$.  Then the following set of sentences is inconsistent.
$$
T_{\calH}\cup \{\neg \psi_i : i\in [m]\} \cup \{\exists x_1\ldots \exists x_n \bigwedge_{1\leq i\neq j\leq n}x_i\neq x_j: n\geq 1\}.
$$
By compactness, there is $K$ such that for all $\calM\models T_{\calH}$ of size at least $K$, $\calM\models \psi_i$ for some $i\in [m]$ (thus if $\calM$ is also finite, then $\calM\in age(\calM_i)$).  Let $\calF$ be the property consisting of the elements in $\calH$ of size at most $K$.  Then $\calH=\calF\cup \bigcup_{i=1}^m age(\calM_i)$.  For all $\ell\geq 2$, the proof of Theorem \ref{thm:tb} shows that $\calH$ is factorial of degree $\ell$ if and only if $\ell$ is the largest integer such that one of the $\calM_i$ has infinitely many components of size $\ell$.
\end{proof}

 \section{A dividing line: mutual algebraicity}\label{sec:ma}

This section contains the remaining ingredients needed for Theorem \ref{thm:mainthm1}.  We proceed by partitioning hereditary properties
 based on the dividing line of mutual algebraicity (see Subsection \ref{ss:prelim} for precise definitions).  The idea is that mutually algebraic properties are ``well behaved,'' allowing a detailed analysis of their structure and speeds, whereas
non-mutually algebraic properties have ``bad behavior'' implying a relatively fast speed. 
Specifically, in Subsection \ref{ss:macase}, we consider the case where $\calH$ is mutually algebraic.  We use structural implications of this assumption to prove the remaining counting dichotomies.  Namely, either $|\calH_n|\geq n^{n(1-o(1))}$, $|\calH_n|=n^{n(1-1/k-o(1))}$ for some integer $k\geq 2$, or $\calH$ is basic, in which case, by Section \ref{sec:eu}, the speed of $\calH$ is asymptotically equal to a sum of the form $\sum_{i=1}^kp_i(n)i^n$ for finitely many rational polynomials $p_1,\ldots, p_k$.   The proofs in Section \ref{ss:macase} rely on Section \ref{sec:tb} along with the fact that a mutually algebraic property is always controlled by finitely many totally bounded properties (see Subsection \ref{ss:macase} for details).

By contrast, in Subsection \ref{ss:lbma}, we show that for any finite relational language $\calL$, if $\calH$ is a non-mutually algebraic hereditary $\calL$-property $\calH$, then $|\calH_n|\geq n^{n(1-o(1))}$.  To prove this, we require a model theoretic result, Theorem \ref{mainthmold}, which relies on results from \cite{ourpaper}.  This theorem describes some properties of large, in fact uncountable, models of $T_{\calH}$.
This allows us to show there is an uncountable model of $T_{\calH}$ which has many distinct finite substructures, yielding the desired lower bound on $|\calH_n|$.  

Our strategy can be seen as a generalization of the strategy employed by Balogh, Bollob\'{a}s, and Weinreich in the graph case \cite{BBW1}.  However, executing this strategy is  significantly more complicated when dealing with relations of arity larger than $2$.  The crucial new ingredient in our proof, Theorem \ref{mainthmold}, required ideas from stability theory.

 \subsection{Preliminaries}\label{ss:prelim}
 In this subsection we give the relevant background on mutually algebraic properties.   We begin with the basic definitions, first introduced in \cite{laskowski2013}.

 \begin{Definition}\label{def:ma}{\em
Given an $\calL$-structure $\calM$, an $\calL$ formula $\phi(\xbar)=\phi(x_1,\ldots, x_s,\ybar)$ and $\abar\in M^{\lg(\ybar)}$,
 $\phi(\xbar,\abar)$ is \emph{$k$-mutually algebraic in $\calM$} if for every partition $[s]=I\cup J$ into nonempty sets $I$ and $J$, 
$$
\calM\models \forall \xbar_I\exists^{<k}\xbar_J\phi(x_1,\ldots, x_s,\abar).
$$ }
\end{Definition}

Note that an $\calL$-structure $\calM$ is totally $k$-bounded if and only if all relations of $\calL$ are $k$-mutually algebraic in $\calM$.

\begin{Definition}\label{def:mam}
{\em An $\calL$-structure $\calM$ is \emph{mutually algebraic} if, for every formula $\psi(\xbar)$ 
there is a finite set $\Delta=\Delta(\xbar; \ybar)$ of $\calL$-formulas, an integer $k$, and parameters $\abar\in M^{|\ybar|}$ such that the following hold.
\begin{enumerate}
\item For every $\varphi(\xbar',\ybar)\in \Delta$, $\varphi(\xbar', \abar)$ is $k$-mutually algebraic in $\calM$ (here, $\xbar'\subseteq \xbar$ is a subsequence of $\xbar$, possibly varying with
$\phi\in\Delta$); and
\item There is a formula, $\theta(\xbar; \ybar)$, which is a boolean combination of elements of $\Delta$, such that 
$\calM\models \forall \xbar (\psi(\xbar)\leftrightarrow \theta(\xbar; \abar))$.
\end{enumerate}}
\end{Definition}

In the definition above, there is no bound on the quantifier complexity of either $\phi$ or of the formulas in $\Delta$.
However, detecting whether or not a structure is mutually algebraic can be seen by looking at quantifier-free formulas.
In fact, as we are working in a finite language $\calL$
 without function symbols,  we have the following characterization.


\begin{Lemma}  \label{4.3}
{\em An $\calL$-structure $\calM$ is mutually algebraic  if and only if for some integers $s,k\geq 0$,
there is a finite set $\Delta=\Delta(\xbar, \ybar)$ of quantifier-free $\calL$-formulas with  $|\ybar|=s$ and parameters $\abar\in M^s$ such that the following hold.
\begin{enumerate}
\item For every $\varphi(\xbar',\ybar)\in \Delta$, $\varphi(\xbar', \abar)$ is $k$-mutually algebraic in $\calM$ (here too, $\xbar'\subseteq \xbar$ is a subsequence of $\xbar$, possibly varying with
$\phi\in\Delta$); and
\item For every relation symbol $R(\xbar')$ of $\calL$, there is a formula $\delta_R(\xbar',\ybar)$, which is a boolean combination of elements of $\Delta$, such that 
$$
\calM\models \forall \xbar' (R(\xbar')\leftrightarrow \delta_R(\xbar', \abar)).
$$
\end{enumerate}}
\end{Lemma}

\begin{proof}  First, assume $\calM$ is mutually algebraic.  By Proposition~4.1 of \cite{Laskowski2009}, every relation is equivalent in $\calM$ to a boolean combination of {\bf quantifier-free}
mutually algebraic formulas.  As there are only finitely many relations in $\calL$, we can choose a finite set $\Delta$ of quantifier-free formulas, and a uniform finite $k$, so that each relation of $\calL$ is equivalent in $\calM$ to a boolean combination of elements of $\Delta$, and such that every element of $\Delta$ is $k$-mutually algebraic in $\calM$.

Conversely, suppose there exist $s,k\in \mathbb{N}$, $\abar\in M^s$, and $\Delta(\xbar;\ybar)$, a set of quantifier-free $\calL$-formulas with $|\ybar|=s$, such that (1) and (2) hold.  Let $MA^*(\calM)$ denote the set of all formulas $\theta(\xbar;\bbar)$ with the property that $\theta(\xbar;\bbar)$ is equivalent in $\calM$ to a boolean combination of mutually algebraic formulas.  By assumption, every relation of $\calL$ is in $MA^*(\calM)$.  Clearly $MA^*(\calM)$ is closed under substituting constants for variables, and under taking boolean combinations.  It is closed under existential quantification
by Proposition~2.7 of \cite{laskowski2013}.  It follows that $\calM$ is mutually algebraic.
\end{proof}



\begin{Definition}\em{
We say a (possibly incomplete) theory $T$  is \emph{mutually algebraic} if every $\calM\models T$ is mutually algebraic. A hereditary $\calL$-property $\calH$ is \emph{mutually algebraic} if $T_{\calH}$ is mutually algebraic.}
\end{Definition}

Observe that every totally bounded $\calL$-structure is automatically mutually algebraic (just take $\Delta$ in Lemma~\ref{4.3} to be the set of all atomic formulas).  We will see that relative to an appropriate change of language, the converse holds as well.  In what follows, a {\em totally bounded frame}, defined precisely below, is an $\calL$-formula encoding conditions (1) and (2) of Lemma~\ref{4.3}.  Given a tuple of variables $\xbar$ and $x\in \cup\xbar$, let $\hat{x}$ denote the tuple obtained from $\xbar$ by deleting $x$.

\begin{Definition} \label{def:frame}  {\em  A {\em totally bounded frame} is a universal 
$\calL$-formula $\theta(\ybar)$ such that the following holds.  There exists $k\in\mathbb{N}$, a finite set $\Delta(\xbar,\ybar)$ of quantifier-free 
$\calL$-formulas,  and for each relation $R(\xbar')\in \calL$, a corresponding formula $\delta_R(\xbar';\ybar)$, which is a boolean combination of elements from $\Delta(\xbar;\ybar)$, so that

$$\theta(\ybar)\vdash  \left(\bigwedge_{\phi\in\Delta}\bigwedge_{x\in\cup \xbar} \forall x \exists^{\le k} \xhat\phi(x,\xhat,\ybar)\right)\wedge\left(\bigwedge_{R\in \calL} \forall\xbar'
[R(\xbar')\leftrightarrow \delta_R(\xbar',\ybar)]\right).$$

}
\end{Definition}

The following Lemma amounts to simply unpacking the definitions, with $(2)\Rightarrow(3)$ being an instance of compactness.

\begin{Lemma}  \label{lem:frame}  The following are equivalent for a (possibly incomplete) $\calL$-theory $T$:
\begin{enumerate}
\item  $T$ is mutually algebraic;
\item  For every $\calM\models T$, there is a totally bounded frame $\theta(\ybar)$ so that $\calM\models\exists\ybar\theta(\ybar)$; and
\item  There is a finite set $\{\theta_j(\ybar_j):1\le j\le m\}$ of totally bounded frames such that $T\vdash\bigvee_{j=1}^m \exists\ybar_j\theta_j(\ybar_j)$.
\end{enumerate}
\end{Lemma}

\subsection{Counting dichotomies for mutually algebraic properties.}\label{ss:macase}
In this subsection we analyze the possible speeds of mutually algebraic hereditary $\calL$-properties.  
The main idea is that, via totally bounded frames $\theta(\ybar)$, any mutually algebraic hereditary $\calL$-property  $\calH$ is essentially controlled by finitely many totally bounded properties
$\calH_\theta^*$, although each of these totally bounded properties will have its own language $\calL_\theta$.
The new language $\calL_\theta$ consists of the constants of $\calL$, $s$ new constant symbols, where $s=|\ybar|$, and a new relation symbol $R_\phi(\xbar')$ for each $\phi(\xbar',\ybar)\in\Delta(\xbar,\ybar)$.

\begin{Definition}  \label{def:constants} {\em  Suppose $\calH$ is a hereditary $\calL$-property and $s\in\mathbb{N}$.  Let $\calL(s):=\calL\cup\{c_1,\dots,c_s\}$, where each $c_i$ is a new constant symbol not in $\calL$.  
For any $\calN\in\calH$ and any $\abar\in N^s$, let $\calN_{\abar}$ denote the natural expansion of $\calN$ to an $\calL(s)$-structure obtained by interpreting each $c_i$ as $a_i$.
Let $\calH(s):=\{\calN_{\abar}:\calN\in\calH, \abar\in N^s\}$. 
}
\end{Definition}

When $\calL(s)$ is clear from context, we will write $\cbar$ to denote $(c_1,\ldots, c_s)$, the tuple of new constant symbols.  Clearly, if $\calH$ is a hereditary $\calL$-property, then $\calH(s)$ is a hereditary $\calL(s)$-property.  
%
Moreover, for any integer $n$, 
$$|\calH_n|\le |\calH(s)_n|\le n^s|\calH_n|.$$

\begin{Definition}
Given a totally bounded frame $\theta(\ybar)$ with $|\ybar|=s$, 
let 
$$
\calH_\theta:=\{\calN\in\calH(s):\calN\models\theta(\cbar)\}.
$$
\end{Definition}

Observe that for any totally bounded frame $\theta(\ybar)$ with $|\ybar|=s$, since $\theta(\ybar)$ is a universal formula, and since $\calH(s)$ is a hereditary $\calL(s)$-property, we have that $\calH_\theta$ is also a hereditary $\calL(s)$-property.

\begin{Lemma} \label{lem:easycount} Suppose $\calH$ is a hereditary $\calL$-property $\calH$  and $\theta(\ybar)$ is a totally bounded frame with $|\ybar|=s$.  For every $n\in \mathbb{N}$, if $n(\theta):=|\{\calM\in\calH_n:\calM\models\exists\ybar\theta(\ybar)\}|$, then
$n(\theta)\le |(\calH_\theta)_n|\le n^s\cdot n(\theta)$. 
\end{Lemma}

\begin{proof}  
The inequalities are obvious, since for any $\calM\in\calH_n$ with $\calM\models\exists\ybar\theta(\ybar)$, there is at least one, and at most $n^s$ many $\abar\in M^s$, such that $\calM_{\abar}\models \theta(\abar)$.
\end{proof}

\begin{Definition} \label{def:lots} {\em  Suppose $\calH$ is a hereditary $\calL$-property and $\theta(\ybar)$ is a totally bounded frame with data $k,s$, $\Delta(\xbar,\ybar)$, and 
$\{ \delta_R(\xbar',\ybar):R(\xbar')\in\calL\}$ as in Definition \ref{def:frame}.  
\begin{itemize}
\item  Let $\Delta_{\cbar}(\xbar):=\{\phi(\xbar',\cbar):\phi(\xbar',\ybar)\in\Delta\}$.
\item  Let $\calL_\theta:=\{\text{constants of }\calL(s)\}\cup\{R_\phi(\xbar'):\phi(\xbar',\cbar)\in\Delta_{\cbar}\}$.
\item  If $\psi$ is a boolean combination of elements of $\Delta_{\cbar}$, let $\psi^*$ denote the $\calL_{\theta}$-formula obtained as follows: for each $\phi(\xbar';\cbar)\in \Delta_{\cbar}$, replace any instance of $\phi(\xbar';\cbar)$ in $\psi$ with $R_{\phi}(\xbar')$.   
\item  Define a function $f:\calH_\theta\rightarrow \{\calL_\theta$-structures$\}$ as follows.  Given $\calM\in\calH_\theta$, let $f(\calM)$ be  the $\calL_\theta$-structure with underlying set $M$, where $c^{f(\calM)}=c^{\calM}$ for all constants $c\in\calL_\theta$, and where $R_\phi^{f(\calM)}:=\{\bbar\in M^{|\xbar'|}:\calM\models\phi(\bbar,\cbar)\}$.
\item Let $\calH^*_\theta:=\{f(\calM):\calM\in\calH_\theta\}$.
\end{itemize}
}
\end{Definition}

We claim that for any $\calM\in \calH_\theta$, any relation $R\in \calL$, and any relation $R_\phi\in \calL_\theta$, $R^\calM$ is $0$-definable in $f(\calM)$ and $R_\phi^{f(\calM)}$ is $0$-definable in $\calM$. Indeed, given a relation $R\in \calL$, $R^\calM=\delta_R(\calM;\cbar)=\delta_R^*(f(\calM);\cbar)$, and given a relation $R_\phi$ of $\calL_\theta$, $R_\phi^{f(\calM)}=\phi(\calM;\cbar)$.  An easy induction on formulas then implies that for any $\ell \in \mathbb{N}$, $\calM$ and $f(\calM)$ have exactly the same $0$-definable subsets of $M^\ell$ (this also uses the facts that $\calL(s)$ and $\calL_\theta$ have the same set of constants, and that $\calM$ and $f(\calM)$ have the same realizations of said constants). These observations make the following lemma straightforward.

\begin{Lemma} \label{lem:hardcount}  Let $\calH$ be a hereditary $\calL$-property and let $\theta(\ybar)$ be a totally bounded frame.
Then the following hold.
\begin{enumerate}
\item The function $f:\calH_\theta\rightarrow\calH_\theta^*$ is bijection. 
\item  For any integer $n$, $f$ maps $(\calH_\theta)_n$ onto $(\calH^*_\theta)_n$, so $|(\calH_\theta)_n|=|(\calH^*_\theta)_n|$.
\item  For every $\calM\in\calH_\theta$, $\calM$ and $f(\calM)$ have the same number of $\sim$-classes (in the sense of Definition~\ref{def:eq}).
\item  $\calH_\theta^*$ is a totally bounded, hereditary $\calL_\theta$-property.
\end{enumerate}
\end{Lemma}

\begin{proof}  Proof of (1):  That $f$ maps $\calH_\theta$ onto $\calH^*_\theta$ is immediate by the definition of $\calH^*_\theta$.  To see that $f$ is injective, suppose $\calM,\calN\in \calH_\theta$ and $f(\calM)=f(\calN)$.
Then clearly, $M=N$, and $c^{\calM}=c^{f(\calM)}=c^{f(\calN)}=c^{\calN}$ for each constant symbol $c\in \calL(s)$.
Fix any relation symbol $R(\xbar')\in\calL(s)$.  Since the relation symbols in $\calL(s)$ are the same as in $\calL$, $R(\xbar')\in \calL$. 
Since $\calM,\calN\models\theta(\cbar)$, we have
$$\calM\models\forall\xbar' (R(\xbar')\leftrightarrow \delta_R(\xbar';\cbar)) \quad \hbox{and}\quad \calN\models\forall\xbar' (R(\xbar')\leftrightarrow \delta_R(\xbar',\cbar)).$$
Thus, 
$R^{\calM}=\{\bbar\in M^{|\xbar'|}:\calM\models\delta_R(\bbar;\cbar)\}=\{\bbar\in M^{|\xbar'|}:f(\calM)\models\delta_R^*(\bbar)\}$,
and dually, $R^{\calN}=\{\bbar\in N^{|\xbar'|}:\calN\models\delta_R(\bbar;\cbar)\}=\{\bbar\in N^{|\xbar'|}:f(\calN)\models\delta_R^*(\bbar)\}$. Since $f(\calM)=f(\calN)$, we have $\delta_R^*(f(\calM);\cbar)=\delta_R^*(f(\calN);\cbar)$. Thus, $R^{\calM}=R^{\calN}$, so the $\calL(s)$-structures $\calM$ and $\calN$ are equal.

Proof of (2): This follows immediately from (1), since for all $\calM\in \calH_\theta$, $\calM$ has underlying set $[n]$ if and only if $f(\calM)$ has underlying set $[n]$.

Proof of (3): This follows from the fact that for every power $\ell$, the subsets of $M^\ell$ defined by quantifier-free $\calL(s)$-formulas in $\calM$ are the same as those defined by quantifier-free $\calL_\theta$-formulas in $f(\calM)$ (see the remarks following Definition \ref{def:lots}).

Proof of (4):  It is clear that $\calH_\theta^*$ is closed under isomorphism, since $\calH_\theta$ is.  To see that $\calH^*_\theta$ is hereditary, choose any $\calM^*\in \calH_\theta^*$ and let $\calN^*\subseteq \calM^*$ be an $\calL_\theta$-substructure of $\calM^*$.  Let $N$ denote the underlying set of $\calN^*$.  We want to show that there is some $\calN\in\calH_\theta$ with $f(\calN)=\calN^*$.  By definition of $\calH_\theta^*$, there is some $\calM\in\calH_\theta$ so that $\calM^*=f(\calM)$.  Let $\calN$ be the $\calL(s)$-substructure of $\calM$ with underlying set $N$ (this is possible, since for each constant $c$ of $\calL(s)$, $c^{\calM}=c^{f(\calM)}=c^{\calN^*}\in N$).  As $\calH_\theta$ is hereditary by
Lemma~\ref{lem:easycount}, $\calN\in\calH_\theta$, so $f(\calN)$ is defined.
We  claim that $\calN^*=f(\calN)$.  Clearly, $\calN^*$ and $\calN$ have the same underlying set, $N$.  For any constant symbol $c$ of $\calL_\theta$, 
$c^{f(\calN)}=c^{\calN}=c^{\calM}=c^{f(\calM)}=c^{\calN^*}$, with the first and third equalities arising by the definition of $f$ and the second and fourth by the definition of being a substructure.
Now choose any relation symbol $R_\phi(\xbar')\in\calL_\theta$.  Say $|\xbar'|=\ell$.  Then
$$
R_\phi^{\calN^*}=R_\phi^{f(\calM)}\cap N^\ell=\phi^{\calM}\cap N^\ell=\phi^{\calN}=R_\phi^{f(\calN)},
$$
where the first equality is from $\calN^*$ being an $\calL_\theta$-substructure of $f(\calM)$ and the underlying sets of $\calN,\calN^*$ both being $N$, the second and fourth equalities are from the definition of $f$, and the third equality is from
$\calN$ being an $\calL(s)$-substructure of $\calM$.  Thus, $\calN^*=f(\calN)$, so $\calN^*\in\calH_{\theta}^*$.  Therefore, we have shown that $\calH_{\theta}^*$ is a hereditary $\calL_\theta$-property.

That $\calH_{\theta}^*$ is totally bounded follows from the fact that for every $\calM\in \calH_\theta$, every $\phi(\xbar',\cbar)\in\Delta_{\cbar}$ is $k$-mutually algebraic in $\calM$ (since $\theta(\ybar)$ is a totally bounded frame), hence every $R_\phi\in\calL_{\theta}$ is $k$-mutually algebraic in $f(\calM)$.  
\end{proof}

We combine Lemmas~\ref{lem:frame} and \ref{lem:hardcount} to get a decomposition of any mutually algebraic property $\calH$.

\begin{Proposition}  \label{prop:decomp}  Suppose $\calH$ is a mutually algebraic hereditary $\calL$-property.   Then there are positive integers $s$ and $m$, such that for each $j\in [m]$, there is a finite relational language $\calL_j$, a totally bounded hereditary $\calL_j$-property $\tildeH_j$, and a map
$\pi_j:\tildeH_j\rightarrow\calH$ satisfying:
\begin{enumerate}
\item  For every $\calM\in\tildeH_j$, the following hold.
\begin{enumerate}
\item  $\calM$ and $\pi_j(\calM)$ have the same universe $M$; and
\item  the number of $\sim$-classes of $\pi_j(\calM)$ is at most the number of $\sim$-classes of $\calM$.
\end{enumerate}
\item  For every integer $n$, the restriction of $\pi_j$ to $(\tildeH_j)_n$ is at most $n^s$-to-one.
\item $\calH=\bigcup_{j=1}^m\{\pi_j(\calM): \calM\in \tildeH_j\}$.
\end{enumerate}
\end{Proposition}

\begin{proof}  
As $T_{\calH}$ is mutually algebraic, Lemma \ref{lem:frame} implies there exists a set $\{\theta_j(\ybar_j):j\in[m]\}$ of  totally bounded frames such that for every $\calM\in\calH$, there is some $j\in [m]$ so that $\calM\models   \exists\ybar_j\theta_j(\ybar_j)$.
Set $s:=\max\{|\ybar_j|:j\in[m]\}$.  For each $j\in [m]$, let $\calL_j:=\calL_{\theta_j}$, $\tildeH_j:=\calH^*_{\theta_j}$, and let $f_j:\calH_{\theta_j}\rightarrow \tildeH_j$ be as in Definition \ref{def:lots} applied to $\theta_j(\ybar_j)$.  Note that by Lemma \ref{lem:hardcount}, each $\tildeH_j$ is a totally bounded hereditary $\calL_j$-property.

For each $j\in [m]$, set $\calH_j:=\{\calM\in\calH:\calM\models\exists\ybar_j\theta_j(\ybar_j)\}$ and define a map $\pi_j:\tildeH_j\rightarrow\calH_j$ as follows:  given $\calM\in \tildeH_j$, let $\pi_j(\calM)$ be the $\calL$-reduct of the $\calL(|\ybar_j|)$-structure, $f_j^{-1}(\calM)$ (note $f_j^{-1}(\calM)$ is well defined since $f_j$ is injective).

Then (1a) holds by definition of $\pi_j$, and (1b) holds by Lemma \ref{lem:hardcount}(3) and the fact that, in general, the number of $\sim$-classes of a structure is non-increasing when taking reducts.  Property (2) follows from Lemmas \ref{lem:easycount} and \ref{lem:hardcount}(2). For (3), it suffices to show that for each $j\in [m]$, $\calH_j=\{\pi_j(\calM):\calM\in \tildeH_j\}$ (since $\calH=\bigcup_{j=1}^m\calH_j$).  To this end, fix $j\in [m]$.  We show first that $\calH_j\subseteq \{\pi_j(\calM):\calM\in \tildeH_j\}$. Suppose $\calN\in \calH_j$. By definition of $\calH_j$, there exists $\abar\in N^{|\ybar_j|}$ such that $\calN\models \theta_j(\abar)$.  Note $\calN_{\abar}\in \calH_{\theta_j}$.  Let $\calN^*=f_j(\calN_{\abar})\in \tildeH_j$.  Since $f_j$ is a bijection (Lemma \ref{lem:hardcount}), $f_j^{-1}(\calN^*)=f_j^{-1}(f_j(\calN_{\abar}))=\calN_{\abar}$.  Clearly the $\calL$-reduct of $\calN_{\abar}$ is $\calN$, and thus $\pi_j(\calN^*)=\calN$. This shows $\calN\in   \{\pi_j(\calM): \calM\in \tildeH_j\}$.  

We now show $\calH_j\supseteq \{\pi_j(\calM):\calM\in \tildeH_j\}$.  Suppose $\calN\in \{\pi_j(\calM): \calM\in \tildeH_j\}$.  Let $\calN^*\in \tildeH_j$ be such that $\calN=\pi_j(\calN^*)$, i.e., $\calN$ is the $\calL$-reduct of $\calN':=f_j^{-1}(\calN^*)$.  By definition of $f_j$, $\calN'\in \calH_{\theta_j}$.  By definition of $\calH_{\theta_j}$, the $\calL$-reduct of $\calN'$ must be in $\calH_j$, so $\calN\in \calH_j$.
\end{proof}

We now combine the results of Sections \ref{sec:eu} and \ref{sec:tb} to prove counting dichotomies for the speed of a mutually algebraic property.  We do this by characterizing their speeds in terms of the speeds of totally bounded properties.

\begin{Theorem}\label{thm:mamainthm1}
Suppose $\calH$ is mutually algebraic.  Then one of the following holds: $\calH$ is basic,  $|\calH_n|=n^{(1-1/\ell-o(1))n}$ for some $\ell\geq 2$, or $|\calH_n|\geq n^{(1-o(1))n}$. 
More specifically, let $\{\tildeH_j:j\in[m]\}$ be as in Proposition~\ref{prop:decomp}.  Then one of the following holds:
\begin{enumerate}[(a)]
\item For some $j\in [m]$, $\tildeH_j$ has infinite components. In this case, $|\calH_n|\geq n^{n(1-o(1))}$.
\item For every $j\in [m]$, $\tildeH_j$ has finite components, and $\ell\geq 2$ is maximal such that for some $j\in [m]$, $\calH^*_j$ has infinitely many components of size $\ell$.  In this case, $|\calH_n|=n^{n(1-1/\ell-o(1))}$.
\item For every $j\in [m]$, $\tildeH_j$ is basic. In this case, $\calH$ is basic.
\end{enumerate}
\end{Theorem}
\begin{proof}
Fix $m,s\in \mathbb{N}$, and $\{\calL_j,\tildeH_j,\pi_j:j\in [m]\}$ be as in Proposition~\ref{prop:decomp}.  For each $j\in [m]$, let $\pi_j(\tildeH_j)=\{\pi_j(\calM): \calM\in \tildeH_j\}$.  We split into cases depending on the complexity of the totally bounded properties $\tildeH_j$.

Suppose first that for each $j\in [m]$, $\tildeH_j$ is basic. Then there is $K\in \mathbb{N}$ such that for each $j\in [m]$, every element of $\tildeH_j$ has at most $K$ distinct $\sim$-classes.  By Proposition~\ref{prop:decomp}(1b), every $\pi_j(\calM)$ has at most $K$ distinct $\sim$-classes.  Since $\calH=\bigcup_{j=1}^m\pi_j(\tildeH_j)$, this implies $\calH$ is basic as well.

Suppose now that for some $j\in [m]$, $\tildeH_j$ has infinite components. Then by Theorem \ref{thm:tb}, $|(\tildeH_j)_n|\geq n^{n(1-o(1))}$, so  by Proposition~\ref{prop:decomp}(2)
$|(\pi_j((\tildeH_j)_n)|\geq n^{-s}\cdot n^{n(1-o(1))}=n^{n(1-o(1))}$.  Thus,
 $|\calH_n|\geq n^{n(1-o(1))}$, since $\pi_j(\tildeH_j)\subseteq \calH$.

We are left with the case where $J:=\{j\in [m]: \tildeH_j\text{ is not basic}\}\neq \emptyset$ and for each $j\in [m]$, $\tildeH_j$ has finite components.  For each $j\in J$, let $w(j)\geq 2$ be the maximum integer such that some element of $\tildeH_j$ has infinitely many components of size $w(j)$.  By Theorem \ref{thm:tb}, $|(\tildeH_j)_n|=n^{n(1-1/w(j)-o(1))}$. Thus if $\ell=\max\{w(j): j\in [m]\}$, these observations and Proposition~\ref{prop:decomp} imply 
$$
n^{-s}\cdot n^{n(1-1/\ell-o(1))}\leq |\calH_n|\leq |J|n^{n(1-1/\ell+o(1))},
$$
which implies $|\calH_n|=n^{n(1-1/\ell+o(1))}$, since $|J|\leq m$ and $s$ are constants.
\end{proof}

Note that Theorem \ref{thm:mamainthm1} together with Corollary \ref{lem:tbb} give us a strong structural understanding of the properties in the factorial range, although we are required to consider properties in other languages.  In forthcoming work, the authors consider characterizations of these properties in terms of the original language, in analogy to the structural characterizations of the factorial range for graph properties from \cite{BBW1}.  This work also shows the gap between the factorial and penultimate range is directly related to \emph{cellularity}, a notion with several interesting model theoretic formulations (see for instance \cite{LM:growth, macpherson_schmerl_1991, schmerl_1990}).

\subsection{A lower bound for non-mutually algebraic hereditary classes}  \label{ss:lbma}

The goal of this subsection is to prove Proposition~\ref{thm:malowerbound}, which shows that if a hereditary $\calL$-property $\calH$ is not mutually algebraic, then $|\calH_n|\geq n^{(1-o(1))n}$.  
 In order to do this, we need to introduce  concepts and quote  results from \cite{ourpaper} that describe the structure of large models of $T_\calH$.
Throughout this subsection, assume that $\calL$ is a finite relational language where every atomic formula has free variables among $\zbar$, with $|\zbar|=r$.  

In a prior version \cite{oldversion} of this article, there was a gap in the proof of Proposition \ref{thm:malowerbound} (Proposition 4.10 there).  Remedying this required additional model theoretic results, namely Lemma 4.21 and Proposition 4.25.  The interested reader without model theoretic training may wish to also refer to  \cite{oldversion}, since the statements of the results there are still correct and use less model theoretic terminology.

Note that whenever $A\subseteq B\subseteq M$, there is a natural projection from $S_{\xbar}(B)$ onto  $S_{\xbar}(A)$ given by restriction, i.e. given $p\in S_{\xbar}(B)$,
let 
$$p\mr{A}:=\{\theta(\xbar,\abar)\in p:\cup\abar\subseteq  A\}.$$

An easy induction on the complexity of quantifier-free formulas shows that for any two distinct $p,q\in S_{\xbar}(A)$, there is some atomic $\calL$-formula
$\alpha(\xbar,\ybar)$ and some $\abar\in A^{|\ybar|}$ such that $\alpha(\xbar,\abar)$ is in the symmetric difference  $p\triangle q$.  Iterating this gives a bound on the size of a separating family for finitely many types.

\begin{Lemma}  \label{lem:distinct}  Suppose $\calM$ is any $\calL$-structure.  For any subset $A\subseteq M$ and any $\xbar\subseteq\zbar$, 
if $\{p_i(\xbar):i\in[m]\}\subseteq S_{\xbar}(A)$ are distinct, then there is $B\subseteq A$, such that $|B|\le mr$ and such that $\{p_i\mr{B}:i\in[m]\}$ are pairwise distinct.
\end{Lemma}

\begin{proof}  Given such a set of types, $\{p_i(\xbar):i\in[m]\}$, we claim by induction on $0\leq j\le m$ that there is $B_j\subseteq A$ of size at most $jr$ such that
$$|\{p_i\mr{B_j}:i\in[m]\}|\ge j,$$
which suffices to prove the Lemma.  For $j=0$, the claim is trivially by taking $B_0=\emptyset$.  So assume $0\leq j<m$, and suppose by induction that $B_j$ is chosen so that $|B_j|\leq jr$ and $|\{p_i\mr{B_j}:i\in[m]\}|\ge j$.  
Let $\ell:=|\{p_i\mr{B_j}:i\in[m]\}|$.  If $\ell\geq j+1$, then taking $B_{j+1}=B_j$ suffices.  However, if $\ell=j$, then as $j<m$, there are distinct $i,i'\in[m]$ such 
that $p_i\mr{B_j}=p_{i'}\mr{B_j}$.  As $p_i\neq p_{i'}$, by the observation above, there is an atomic $\alpha(\xbar,\abar)\in p_i\triangle p_{i'}$.
Setting $B_{j+1}:=B_j\cup\{\cup \abar\}$ then suffices (note $|B_{j+1}|\leq |B_j|+r\leq (j+1)r$).
\end{proof}

\begin{Definition}\label{def:supp2}{\em 
Suppose $\calM$ is an $\calL$-structure.  An  \emph{infinite array} in $\calM$ is any set of the form $\{\dbar_i: i\in \mathbb{N}\}\subseteq M^k$ for some $k\geq 1$, such that $(\cup \dbar_i)\cap(\cup \dbar_j)=\emptyset$ for distinct $i,j\in\mathbb{N}$.

Given $A\subseteq M$ and $p\in S_{\xbar}(A)$, we say $p$ \emph{supports an infinite array} in $\calM$ there is an infinite array of realizations of $p$ in $\calM$. }
\end{Definition}

\begin{Definition}  \label{def:supp1}{\em  An $\calL$-structure $\calU$ is {\em $\aleph_1$-saturated\/}\footnote{This notion usually refers to realizations of types involving quantifiers, but here we consider only quantifier-free types.}  if, for every countable set $A\subseteq U$, for every $\xbar\subseteq\zbar$, and for every 
$p(\xbar)\in S_{\xbar}(A)$, $\calU$ {\em realizes $p$}, i.e., there is $\bbar\in U^{|\xbar|}$ such that $\calU\models\theta(\bbar,\abar)$ for every $\theta(\xbar,\abar)\in p$.
}
\end{Definition}

Suppose $\calU$ is $\aleph_1$-saturated.  A type $\pp\in S_{\xbar}(U)$ is called a {\em global type}.  Such types $p$ are typically not realized in $\calU$, but for any countable set $A\subseteq U$, the restriction $\pp|A$ is realized in $\calU$.  In \cite{ourpaper}, the authors identify three important classes of global types.

\begin{Definition} {\em  Suppose $\calU$ is $\aleph_1$-saturated and $\xbar\subseteq\zbar$.

\vspace{2mm}
\noindent$\bullet$  $\Supp_{\xbar}(\calU):=\{\pp\in S_{\xbar}(U): \text{for every countable }A\subseteq U, \text{ there is }\bbar\in (U\setminus A)^{|\xbar|}\text{ realizing } \pp|A\}$.

\vspace{2mm}
\noindent$\bullet$ $\QMA_{\xbar}(\calU)=\{\pp\in \Supp_{\xbar}(\calU):\pp\text{ contains a mutually algebraic formula of the form }\theta(\xbar,\abar)\}$.

\vspace{2mm}
\noindent$\bullet$  A type $\pp\in\Supp_{\xbar}(\calU)$ is {\em array isolated} if there is some $\theta(\xbar,\abar)\in\pp$ such that $\pp$ is the unique element of $\Supp_{\xbar}(\calU)$ containing $\theta(\xbar;\abar)$.
}
\end{Definition}

Global types $\pp\in\Supp_{\xbar}(\calU)$ are called {\em supportive types}.  This is because an easy compactness argument shows that for a global type $\pp$, $\pp\in \Supp_{\xbar}(\calU)$ if and only
if for every countable $A\subseteq U$, $\pp|A$ supports an infinite array in $\calU$.


In Theorem~6.1 of \cite{ourpaper}, the authors give several
equivalents of mutual algebraicity in a finite, relational language.

\begin{Theorem}[Theorem 6.1 of \cite{ourpaper}]    \label{mainthmold}
Suppose $\calU$ is an $\aleph_1$-saturated $\calL$-structure.  Then the following are equivalent.
\begin{enumerate}
\item  $Th(\calU)$ is mutually algebraic;
\item  For all $\xbar\subseteq\zbar$, $\Supp_{\xbar}(\calU)$ is finite;
\item  For all $\xbar\subseteq\zbar$, $\QMA_{\xbar}(\calU)$ is finite;
\item  For all $\xbar\subseteq\zbar$, every $\pp\in\Supp_{\xbar}(\calU)$ is array isolated.
\end{enumerate}
\end{Theorem}

Moreover, the proof of Theorem \ref{mainthmold} in \cite{ourpaper} shows the equivalence of (2), (3), and (4) holds locally for each tuple $\xbar\subseteq \zbar$.  A crucial idea in this section will be, given a non-mutually algebraic hereditary property $\calH$, to consider the {\em shortest}  $\xbar\subseteq\zbar$ for which Clause (3) of 
Theorem~\ref{mainthmold} fails in some $\aleph_1$-saturated $\calU\models T_{\calH}$ (e.g. this will occur in the proof of Proposition \ref{prop:notMA}).  In a related vein, our next lemma, Lemma \ref{5.9}, considers what one can deduce about types in the variables $\xbar$, when every proper subtuple $\xbar'\subsetneq \xbar$ satisfies Clause (3) above.  To state Lemma \ref{5.9}, we first require  a definition and some further results from \cite{ourpaper}.

%

\begin{Definition}\label{def:product}{\em 
Suppose $\calU$ is an $\aleph_1$-saturated $\calL$-structure and $\pp(\xbar)\in \Supp_{\xbar}(\calU)$ and $\qq(\ybar)\in\Supp_{\ybar}(\calU)$ are array isolated, global types in disjoint variables $\xbar,\ybar\subseteq \zbar$.  The \emph{free product} $\pp\otimes\qq$ is defined to be the set of formulas $\theta(\xbar,\ybar;\abar)$, such that for some countable $\calM\prec \calU$ with $\cup \abar\subseteq M$, some $\cbar\in U^{|\xbar|}$ realizing $\pp\mr{M}$, and some $\dbar\in U^{|\ybar|}$ realizing $\qq\mr{M\cbar}$, it holds that $\calU\models \theta(\cbar,\dbar;\abar)$.

}
\end{Definition}

In \cite{ourpaper}, the authors showed that in the notation of Definition \ref{def:product}, $\pp\otimes\qq\in \Supp_{\xbar\ybar}(\calU)$, and moreover, $\theta(\xbar,\ybar,\abar)\in\pp\otimes\qq$ if and only if  $\calU\models\theta(\cbar,\dbar,\abar)$ for every countable $\calM\preceq\calU$ with $\cup \abar\subseteq M$, for every $\cbar\in U^{|\xbar|}$ realizing $\pp\mr{M}$,
and for every $\dbar\dbar\in U^{|\ybar|}$ realizing $\qq\mr{M\cbar}$.

We are now ready to state Lemma \ref{5.9}.  Its proof arises from simply relativizing the proof of Proposition~5.9 of \cite{ourpaper} to a subsequences $\xbar\subseteq\zbar$ with only finitely many supportive types.


\begin{Lemma}[{\it cf.} Proposition 5.9 of \cite{ourpaper}]  \label{5.9}
Let $\calU$ be an $\aleph_1$-saturated $\calL$-structure.
Suppose $\xbar\subseteq\zbar$ is a non-empty subtuple, and $\QMA_{\xbar'}(\calU)$ is finite for all $\xbar'\subsetneq\xbar$.  Then the following hold.
\begin{enumerate}
\item  For all $\xbar'\subsetneq\xbar$, every $\pp\in\Supp_{\xbar'}(\calU)$ is a free product of at most $r$ types 
from $\bigcup\{\QMA_{\xbar''}(\calU):\xbar''\subseteq\xbar'\}$.
In particular $\Supp_{\xbar'}(\calU)$ is finite.
\vspace{2mm}
\item  Every $\pp\in\Supp_{\xbar}(\calU)\setminus \QMA_{\xbar}(\calU)$ is a free product of at most $r$ types 
 from $\bigcup\{\QMA_{\xbar'}(\calU):\xbar'\subsetneq\xbar\}$.  Thus,
$\Supp_{\xbar}(\calU)\setminus\QMA_{\xbar}(\calU)$ is finite as well.
\end{enumerate}
\end{Lemma}

\begin{proof}  For (1), this is a direct consequence of the proof of Proposition~5.9 of \cite{ourpaper}. 
%

We now show (2).  Choose $\pp\in\Supp_{\xbar}(\calU)\setminus \QMA_{\xbar}(\calU)$, and let 
$\calM\preceq\calU$ be any countable submodel.  Choose a realization
$\cbar$ of $\pp\mr{M}$ in $\calU$, and fix a maximal mutually algebraic decomposition $\cbar=\cbar_1\wedge\dots\wedge\cbar_s$ of $\cbar$ (i.e. for each $1\leq i\leq s$, $\cbar_i$ satisfies a quantifier-free mutually algebraic formula with parameters from $M$, and $s$ is as small as possible).  Since $\qftp(\cbar/M)$ is not mutually algebraic,
$s>1$.  For each $j$, let $\qq_j$ be the global type extending $\qftp(\cbar_j/M)$.  Then run the argument from the proof of Proposition~5.9 to conclude that if
$\pp\neq \qq_1\otimes\dots\otimes\qq_s$, then $\qftp(\cbar/M)$ would contain a mutually algebraic formula $\theta(\xbar,\mbar)$, contradicting $\pp\not\in\QMA_{\xbar}(\calU)$.
\end{proof}

Observe that when $|\xbar|=1$, Lemma \ref{5.9} (1) is vacuous, and (2) is immediate (since in that case, $\Supp_{\xbar}(\calU)=\QMA_{\xbar}(\calU)$).  Further, we point out that if $Th(\calU)$ is mutually algebraic, then Lemma \ref{5.9} follows immediately from Proposition~5.9  and Theorem~6.1 of \cite{ourpaper}.

With the technical machinery above, our next step is to describe an {\em indiscernible grid} (Definition \ref{def:grid} below) and show that if $\calH$ is any non-mutually algebraic hereditary property, then
there is an indiscernible grid $\calN\models T_{\calH}$ (Proposition \ref{prop:notMA} below).  From this $\calN$, we will extract a family of finite substructures that we will will use to prove Proposition~\ref{thm:malowerbound}.

\begin{Definition} \label{def:grid1} {\em  Suppose $\calM$ is an infinite $\calL$-structure, $\xbar\subseteq \zbar$ is a non-empty subtuple, and $P=\{p_i(\xbar):i\in \mathbb{N}\}$ is a set of distinct types from $S_{\xbar}(M)$.    A {\em grid for $\calM$ and $P$} is an $\calL$-structure $\calN$ whose universe has the form $N=M\cup\bigcup\{\cup\bbar_{i,q}:i\in\mathbb{N},q\in\Q\}$, such that the following holds.
\begin{enumerate}
\item For each $i\in \mathbb{N}$ and $q\in \mathbb{Q}$, $\bbar_{i,q}\in (N\setminus M)^s$, and $|\cup \bbar_{i,q}|=s$, where $s=|\xbar|$. 
\item For each $i\in \mathbb{N}$, $\{\bbar_{i,q}: q\in \mathbb{Q}\}$ is a set of realizations of $p_i(\xbar)$.
\item Each $b\in N\setminus M$ is contained in exactly one $\bbar_{i,q}$.
\end{enumerate}
}
\end{Definition}

Suppose $\calN$ is a grid for $\calM$ and $P$, in the notation of Definition \ref{def:grid1}. Observe that (1) implies that for each $i\in \mathbb{N}$, and every $x\neq x'$ from the tuple $\xbar$, we have that the formula $x\neq x'$ is in $p_i(\xbar)$.  Given $\sigma\in Aut(\mathbb{Q},\leq)$, note that $\sigma$ induces a permutation $\sigma^*$ on $N$ as follows:
$\sigma^*(m)=m$ for all $m\in M$ and $\sigma^*(\bbar_{i,q})=\bbar_{i,\sigma(q)}$.

\begin{Definition}\label{def:grid}  {\em Suppose $\calN$ is a grid for $\calM$ and $P$, in the notation of Definition \ref{def:grid1}.  
\begin{enumerate}
\item We say $\calN$ is an {\em indiscernible grid for $\calM$ and $P$} if for every $\sigma\in Aut(\Q,\le)$, $\sigma^*$ is an $\calL$-automorphism of $\calN$.
\item Assuming $\calN$ is an indiscernible grid for $\calM$ and $P$, a {\em hybrid tuple} is an $s$-tuple $\dbar\in (N\setminus M)^s$ such that $\dbar$ is not a permutation of any
$\bbar_{i,q}$.  A \emph{hybrid type} is an element of $S_{\xbar}(M)$ of the form $\qftp(\dbar/M)$, for some hybrid tuple $\dbar$.  

\end{enumerate}
}
\end{Definition}

\begin{Remark}  {\em  Note that if $s=1$, then there are no hybrids.  
This is what makes the $s=1$ case much easier than the general case in what follows.}
\end{Remark}



\begin{Proposition} \label{prop:notMA}   Suppose $T$ is any universal $\calL$-theory that is not mutually algebraic.
Then there are $1\leq s\le r$,  an integer $e$,  a countable $\calM\models T$, and an infinite set $P=\{p_i(\xbar):i\in\mathbb{N}\}\subseteq S_{\xbar}(M)$ with $|\xbar|=s$,
such that each $p_i(\xbar)$ contains a mutually algebraic formula $\theta_i(\xbar)$.  Moreover:

\begin{enumerate}
\item  There exists an indiscernible grid $\calN$ for $\calM$ and $P$ such that $\calN\models T$; and
\item  $|\{\qftp(\dbar/M):\dbar\in N^s$ a hybrid tuple$\}|\le e$.
\end{enumerate}
\end{Proposition}


\begin{proof}  As $T$ is not mutually algebraic, use  Theorem~\ref{mainthmold}(3) of \cite{ourpaper} to choose $1\leq s\le r$ least such that there exists $\xbar\subseteq \zbar$ with $|\xbar|=s$, and an 
$\aleph_1$-saturated $\calU\models T$ with infinitely many distinct global types $S=\{\pp_i(\xbar)\in \QMA_{\xbar}(\calU):i\in\mathbb{N}\}$.
 Fix such a $\calU$, $\xbar$, and $S=\{\pp_i(\xbar)\in \QMA_{\xbar}(\calU):i\in\mathbb{N}\}$.  
  By the minimality of $s$, for any proper subsequence $\xbar'\subsetneq\xbar$,
there are only finitely many $\qq(\xbar')\in\QMA_{\xbar'}(\calU)$.  Thus, by eliminating at most finitely many types from $S$, we may assume
that for each $i\in\mathbb{N}$, $(x\neq x')\in\pp_i(\xbar)$ for all distinct $x,x'\in\xbar$.  Since each $\pp_i(\xbar)\in\Supp_{\xbar}(\calU)$,  $(x\neq m)\in \pp_i$ 
for all $x\in \xbar$ and $m\in \calU$.  Let $\calM\preceq \calU$ be any countable, elementary substructure such that the restrictions 
 $p_i(\xbar):=\pp_i\mr{M}$ are pairwise distinct 
  and set $P:=\{p_i(\xbar): i\in \mathbb{N}\}$.

To prove (1) holds, choose, for each $i\in\mathbb{N}$, an
infinite array $\{\dbar_{i,\ell}:\ell\in\mathbb{N}\}$ of realizations of $p_i(\xbar)$ in $\calU$.
By the pigeon-hole principle, we may assume, possibly after reindexing, that
$\dbar_{i,\ell}$ and $\dbar_{i',\ell'}$ are disjoint unless $(i,\ell)=(i',\ell')$.  
Let $\calM'\preceq\calU$ be a countable elementary substructure containing $M\cup\bigcup\{\dbar_{i,\ell}:i,\ell\in\mathbb{N}\}$.
Visibly, $\calM'$ contains a grid for $\calM$ and $P$, but it might not be indiscernible.
We obtain an indiscernible grid by compactness:  Let $\calL^*$ be $\calL$, adjoined with new constant symbols $\{\cbar_m:m\in M\}\cup\{\cbar_{i,q}:i\in\mathbb{N},q\in\Q\}$ (each $\cbar_{i,q}$ is an $s$-tuple of new 
constant symbols) and let $T^*$ be the $\calL^*$-theory asserting:
\begin{itemize}
\item  The elementary diagram of $\calM$;
\item  For each $i\in\mathbb{N}$, each $\cbar_{i,q}$ realizes $p_i(\xbar)$, 
\item For distinct $(i,q),(i',q')\in \mathbb{N}\times \mathbb{Q}$, $\cbar_{i,q}$ and $\cbar_{i',q'}$ are disjoint;
\item  For each $\sigma\in Aut(\Q,\le)$, $\sigma^*$ is an $\calL$-automorphism of $M\cup\bigcup\{\bigcup\cbar_{i,q}:i\in\mathbb{N},q\in\Q\}$.
\end{itemize}

Arguing by induction on the number of $\sigma$'s mentioned,  using Ramsey's theorem one shows  that every finite subset of $T^*$ can be realized in an $\calL^*$-expansion of $\calM'$.  
By compactness, $T^*$ has a model $\calM^*$.   For each $i\in \mathbb{N},q\in \mathbb{Q}$, let $\bbar_{i,q}$ be the realization of $\cbar_{i,q}$ in $\calM^*$.  Set $N=M\cup\bigcup\{\bigcup\bbar_{i,q}:i\in\mathbb{N},q\in\Q\}$, and $\calN^*=\calM^*[N]$.  Finally, let $\calN$ be the $\calL$-reduct of $\calN^*$. 

To prove (2), let $\calU\succeq \calM$ be an $\aleph_1$-saturated elementary extension.  By the minimality of $s$, for every proper $\xbar'\subsetneq\xbar$,
there are only finitely many $q(\xbar')\in S_{\xbar'}(M)$ that contain a mutually algebraic formula and support an infinite array.  It follows that
$\QMA_{\xbar'}(\calU)$ is finite for each proper $\xbar'\subsetneq\xbar$. Hence, $\calQ:=\bigcup\{\QMA_{\xbar'}(\calU):\xbar'\subsetneq\xbar\}$ is finite as well.

We claim that for every hybrid tuple $\dbar\in (N\setminus M)^s$, $\qftp(\dbar/M)$ is the restriction to $M$ of a free product of types from $\calQ$.
To see this, fix $\dbar\in (N\setminus M)^s$ that is not contained in any $\bbar_{i,q}$.  Clearly, $\qftp(\dbar/M)$ supports an infinite array (in fact, $\calN$ contains an infinite array of realizations of this type), but the indiscernibility demonstrates that $\qftp(\dbar/M)$ cannot contain a mutually algebraic formula.
Thus, the global extension $\pp$ of $\qftp(\dbar/M)$ is in $\Supp_{\xbar}(\calU)\setminus\QMA_{\xbar}(\calU)$.  So, by Lemma~\ref{5.9}(2), 
$\pp$ is the free product of at most $r$ global types from $\calQ$.  Since $\calQ$ is finite, this shows (2).

\end{proof}


\begin{Lemma} \label{lem:template} Suppose $\calH$ is a hereditary $\calL$-property that is not mutually algebraic.  Then there are positive integers $s\le r$ and $e$ such that for every positive integer $L$, there is $\calN_L\models T_{\calH}$ and quantifier-free formulas $\{\phi_i(\xbar,\abar):i\in[L]\}$ such that $N_L$ can be partitioned as 
$$
N_L=A_L\cup\{\cup \bbar_{i,q}:i\in [L],q\in\Q\},
$$
  so that the following hold.
\begin{enumerate}
\item For each $i\in [L]$ and $q\in \mathbb{Q}$, $\bbar_{i,q}\in N_L^s$ and $|\cup \bbar_{i,q}|=|\xbar|=s$.
\item  $|A_L|\le (L+e)r+w$, where $w$ is the number of constants from $\calL$.
\item For each constant $c$ of $\calL$, $c^{\calN_L}\in A_L$.
\item  $\cup \abar=A_L$.
\item  For all $i\in [L]$, the following hold.
\begin{enumerate}
\item  For all $q\in \Q$, $\calN_L\models\phi_i(\bbar_{i,q};\abar)$; and
\item  If  $\dbar\in (N_L)^s$ and $\calN_L\models\phi_i(\dbar;\abar)$, then  $\dbar$ is a permutation of some $\bbar_{i,q}$.
\end{enumerate}
\end{enumerate}
\end{Lemma}

\begin{proof}  As $T_{\calH}$ is not mutually algebraic, Proposition~\ref{prop:notMA} implies there exist positive integers $s\le r$ and $e$, a countable $\calM\models T_{\calH}$, an infinite set $P=\{p_i(\xbar):i\in \mathbb{N}\}\subseteq S_{\xbar}(M)$, where $|\xbar|=s$, and an indiscernible grid $\calN$ for $\calM$ and $P$, with
universe 
$$
N=M\cup\bigcup\{\cup\bbar_{i,q}:i\in\mathbb{N},q\in\Q\},
$$
such that there are at most $e$ many hybrid types realized in $\calN$ over $M$.  Let $\{q_j:j\in[e]\}$ be the hybrid types over $M$ realized in $\calN$.

Let $A_L$ be a minimal subset of $\calM$ containing $c^{\calN}$ for every constant $c$, and such that the restrictions of the types $\{p_i(\xbar): i\leq L+e\}$ to $A_L$ are pairwise distinct.  In light of Lemma~\ref{lem:distinct}, such an $A_L$ can be found of cardinality at most $(L+e)r+w$.

Let $\calN_L$ be the substructure of $\calN$ with universe $N_L=A_L\cup\{\bbar_{i,q}:i\in [L],q\in\Q\}$.
As $A_L$ is finite, every complete quantifier-free type over $A_L$ can be described by a single formula.  For each $i\in[L]$, let $\phi_i(\xbar;\abar)$ be the formula
describing the restriction $p_i\mr{A_L}$.  Since each $\bbar_{i,q}$ realizes $p_i$, we have $\calN_L\models\phi_i(\bbar_{i,q};\abar)$ for every $q\in\Q$.  By construction, we now have properties (1)-(3) as well as (4a).  

We just need to verify (4b).  To this end, fix $\dbar\in (N_L)^s$ such that $\calN_L\models\phi_i(\dbar;\abar)$ for some $i\in [L]$.  Since $p_i$ implies $x\neq m$ for every $x\in\xbar$ and $m\in M$, and since $\phi_i(\xbar;\abar)$ describes $p_i\mr{A_L}$, we must have $\dbar\in (N_L\setminus A_L)^s$. By definition of $N_L$ and $A_L$, this implies $\dbar\in (N\setminus M)^s$.  By our choice of $A_L$, for each $j\in [e]$, $p_i\mr{A_L}\neq q_j\mr{A_L}$, so $\qftp(\dbar/M)\notin \{q_j: j\in [e]\}$.  Thus $\dbar\in (N\setminus M)^s$ is not a hybrid tuple, and consequently must be a permutation of $\bbar_{i',q'}$ for some $i'\in \mathbb{N}$ and $q'\in \mathbb{Q}$. Since $\dbar\in N_L^s$, we must have $i'\in [L]$.  Finally, because $i'\in [L]$, $\dbar\models p_{i}\mr{A_L}$, and $p_i\mr{A_L}\neq p_{j}\mr{A_L}$ for all $j\in [L]\setminus\{i\}$, we must have that $i=i'$. 
\end{proof}

\vspace{2mm}
\begin{Proposition}\label{thm:malowerbound}
For any hereditary $\calL$-property $\calH$, if $\calH$ is not mutually algebraic, then  $|\calH_n|\geq n^{(1-o(1))n}$.
\end{Proposition}
\begin{proof}

Assume $\calH$ is a hereditary $\calL$-property which is not mutually algebraic.  Let $c$ denote the number of constants of $\calL$.  Apply Lemma~\ref{lem:template} to obtain $1\leq s\leq r$ and $e$.   We show that for every integer $t\geq 1$,  $|\calH_n|\geq n^{(1-\frac{r+1}{ts+r}-o(1))n}$.  This implies  $|\calH_n|\ge n^{(1-o(1))n}$.

Fix an integer $t\ge 1$ and choose $n\gg tre$.  We will construct a special $\calG\in \calH_n$, and then show there are many distinct elements of $\calH_n$, each isomorphic to $\calG$.   

Set $L=\lfloor \frac{n-er-c}{st+r}\rfloor$.  Let $\calN_L$, $A_L$, $\{\phi_i(\xbar;\abar): i\in [L]\}$ and $\{\bbar_{i,q}: i\in [L], q\in \mathbb{Q}\}$ be as in the conclusion of Lemma~\ref{lem:template}.  Let $B=\bigcup\{\cup\bbar_{i,j}:i\in[L], j\in [t]\}$.  Note that $|B|=Lst$, so $|A_L\cup B|\leq Lst+(L+e)r+c\leq n$.  Choose $X\subseteq N_L\setminus (A_L\cup B)$ so that $|A_L|+|B|+|X|=n$.  

Define $\calG:=\calN_L[A_L\cup B\cup X]$, and observe that $\calG\in \calH_n$.  For each $i\in [L]$, let $\theta_i(\xbar)$ be the formula $\phi_i(\xbar,\abar)\wedge\bigwedge_{x\in\xbar,a\in X\cup A_L} x\neq a$.  Note each $\theta_i(\xbar)$ has parameters from $A_L\cup X$, and $\calG\models\bigwedge_{j=1}^t \theta_i(\bbar_{i,j})$.  Further, for all $\dbar\in G^s$,  if $\calG\models\theta_i(\dbar)$, then there is $j\in [t]$ such that $\dbar$ is a permutation of $\bbar_{i,j}$.

Let $\theta(\xbar)=\bigvee_{i=1}^L\theta_i(\xbar)$.  Clearly $\calG\models \bigwedge_{i=1}^L\bigwedge_{j=1}^t\theta(\bbar_{ij})$.  Further, for all $\dbar\in G^s$, if $\calG\models \theta(\dbar)$, then $\dbar$ is the permutation of $\bbar_{i,j}$ for some $i\in [L]$ and $j\in [t]$.  We now give a procedure for constructing many distinct $\calL$-structures, each isomorphic to $\calG$.

\begin{enumerate}
\item Choose an equipartition $\calW=\{W_1,\ldots, W_{Lt}\}$ of $B$ into $Lt$ pieces, each of size $s$.  For each $i\in [Lt]$, let $\wbar_i$ denote the tuple enumerating $W_i$ in increasing order.
\item Choose an equipartition $\calQ=\{Q_1,\ldots, Q_L\}$ of $\calW$ into $L$ pieces, each of size $t$.  For each $1\leq i\leq L$, let $1\leq \alpha^i_1<\ldots<\alpha^i_t\leq Lt$ be such that $Q_i=\{W_{\alpha^i_1},\ldots, W_{\alpha^i_t}\}$. 
\item Let $f_{\calW,\calQ}:[n]\rightarrow [n]$ be the function which fixes $A_L\cup X$ pointwise, and such that, for each $1\leq i\leq L$ and $1\leq j\leq t$, $f_{\calW,\calQ}(\bbar_{ij})=\wbar_{\alpha^i_j}$. 
\item Let $\calG_{\calW,\calQ}$ be the $\calL$-structure with universe $[n]$ so that $f_{\calW,\calQ}:\calG\rightarrow\calG_{\calW,\calQ}$ is an $\calL$-isomorphism.
\end{enumerate}

We now show that if $\calW$ and $\calW'$ are distinct choices from step (1), then for any respective choices of $\calQ$ and $\calQ'$ in step (2), we have $\calG_{\calW,\calQ}\neq \calG_{\calW',\calQ'}$.  Indeed, let $\calW\neq \calW'$ be distinct equipartitions of $B$. Then there is some $W\in \calW$ with $W\notin \calW'$.  Let $\wbar$ enumerate $W$ in increasing order.  Then by construction, there are $i\in [L]$, $j\in [t]$, such that $\wbar=f_{\calW,\calQ}(\bbar_{ij})$, and thus $\calG_{\calW,\calQ}\models \theta(\wbar)$.  However, $W\notin \calW'$ implies that no permutation of $\wbar$ can be equal to $f_{\calW',\calQ'}(\bbar_{i'j'})$, for any $i'\in [L]$, $j'\in [t]$.  Consequently, $\calG_{\calW',\calQ'}\models \neg \theta(\wbar)$.  Thus $\calG_{\calW,\calQ}\neq \calG_{\calW',\calQ'}$.  

We now show that for any fixed choice of $\calW=\{W_1,\ldots, W_t\}$ from step (1), if $\calQ\neq \calQ'$ are distinct choices in step (2), then we have $\calG_{\calW,\calQ}\neq \calG_{\calW,\calQ'}$.  Indeed, suppose $\calQ=\{Q_1,\ldots, Q_t\}$ and $\calQ'=\{Q_1',\ldots, Q_t'\}$ are distinct equipartitions of $\calW$. Then there is some $1\leq i\neq i'\leq t$ for which $Q_i\cap Q_i'\neq \emptyset$, say $u\in Q_i\cap Q_i'$. Let $\wbar_u$ enumerate $W_u$ in increasing order.  By construction, $\wbar_u=f_{\calW,\calQ}(\bbar_{ij})=f_{\calW,\calQ'}(\bbar_{i'j'})$ for some $j,j'\in [t]$.   Thus $\calG_{\calW,\calQ}\models \phi_i(\wbar_u;\abar)$ while $\calG_{\calW, \calQ'}\models \phi_{i'}(\wbar_u;\abar)$.  Since $i\neq i'$, $\phi_i(\xbar;\abar)\vdash \neg\phi_{i'}(\xbar;\abar)$, so we must have $\calG_{\calW,\calQ}\neq \calG_{\calW,\calQ'}$.

Thus $|\calH_n|$ is at least the number of equipartitions of $[|B|]=[Lst]$ into $Lt$ pieces, times the number of equipartitions of $[Lt]$ into $L$ pieces.  By Lemma \ref{lem:lb} we obtain the following, where $n'=Lst$ (note $n'\gg s,r,e,c$).
\begin{align*}
|\calH_n|\geq (n')^{(1-1/s-o(1))n'}(n'/s)^{(1-1/t)(n'/s)}=(n')^{(1-1/ts-o(1))n'}&=n^{(1-1/ts-o(1))n'}\\
&=n^{(1-1/ts-o(1))\frac{tsn}{ts+r}}\\
&=n^{(\frac{ts}{ts+r}-\frac{1}{ts+r}-o(1))n}\\
&=n^{(1-\frac{r+1}{ts+r}-o(1))n}.
\end{align*}
We have shown that for all $t\geq 1$, $|\calH_n|\geq n^{(1-\frac{r+1}{ts+r}-o(1))n}$.  Consequently, $|\calH_n|\geq n^{(1-o(1))n}$. 
\end{proof}

 \section{Proof of Theorem \ref{thm:mainthm1} and Minimal Properties in Range 1}
 In this section we bring together what we have shown to prove Theorem \ref{thm:mainthm1}.  We then characterize the minimal properties of each speed in range 1.  
 
 \vspace{2mm}
 
  \noindent{\bf Proof of Theorem \ref{thm:mainthm1}.}
If $\calH$ is basic, then by Theorem \ref{thm:efthm} there are finitely many rational polynomials $p_1,\ldots, p_k$ such that for all sufficiently large $n$, $|\calH_n|=\sum_{i=1}^kp_i(x)i^n$, so case (1) holds.  So assume $\calH$ is not basic.  Observe that by definition, this implies $r\geq 2$.
 
If $\calH$ is also not mutually algebraic, then Theorem \ref{thm:malowerbound} implies $n^{n(1-o(1))}\leq |\calH_n|$ and case (3) holds.  We are left with the case when $\calH$ is mutually algebraic and not basic.  By Theorem \ref{thm:mamainthm1}, either $|\calH_n|\geq n^{n(1-o(1))}$ (so case (3) holds), or there is an integer $k\geq 2$ such that $|\calH_n|=n^{n(1-1/k-o(1))}$ (case (2) holds). 
 \qed
 \vspace{5mm}
 
 An immediate corollary of the proof of Theorem \ref{thm:mainthm1} is the converse of Theorem \ref{thm:efthm}.

\begin{Corollary}\label{cor:basic}
$\calH$ is basic if and only if there is $k\geq 1$ and rational polynomials $p_1,\ldots, p_k$ such that for sufficiently large $n$, $|\calH_n|=\sum_{i=1}^kp_i(x)i^n$.
\end{Corollary}
 
Now that we have Corollary \ref{cor:basic}, we can characterize the minimal properties of each speed in range (1).
 Suppose $\calH$ is a hereditary $\calL$-property.  A \emph{strict subproperty of $\calH$} is any hereditary $\calL$-property $\calH'$ satisfying $\calH'\subsetneq \calH$.  We say $\calH$ is \emph{polynomial} if asymptotically $|\calH_n|=p(n)$ for some rational polynomial $p(x)$.  In this case, the \emph{degree of $\calH$} is the degree of $p(x)$.  We say $\calH$ is \emph{exponential} if its speed is asymptotically equal to a sum of the form $\sum_{i=1}^{\ell}p_i(n)i^n$, where the $p_i$ are rational polynomials and $\ell \geq 2$.  In this case, the \emph{degree of $\calH$} is $\ell$.  A polynomial hereditary $\calL$-property $\calH$ is \emph{minimal} if every strict subproperty of $\calH$ is polynomial of strictly smaller degree.  An exponential hereditary $\calL$-property $\calH$ is \emph{minimal} if every strict subproperty of $\calH$ is either exponential of strictly smaller degree or polynomial.

\begin{Theorem}
Suppose $\calH$ is a non-trivial hereditary $\calL$-property.  
\begin{enumerate}
\item $\calH$ is a minimal polynomial property of degree $k\geq 0$ if and only if $\calH=age(\calM)$ for some countably infinite $\calL$-structure with one infinite $\sim$-class and exactly $k$ elements contained in finite $\sim$-classes.
\item $\calH$ is a minimal exponential property of degree $\ell \geq 2$ if and only if $\calH=age(\calM)$ for some countably infinite $\calL$-structure with $\ell$ infinite $\sim$-classes and no finite $\sim$-classes.
\end{enumerate}
\end{Theorem}
\begin{proof} Fix a hereditary $\calL$-property $\calH$ which is polynomial of degree $k\geq 0$ (respectively exponential of degree $\ell\geq 2$) and minimal.  By Corollary \ref{cor:basic}, $\calH$ is basic. By Corollary \ref{cor:polcount1} there are finitely many countably infinite $\calL$-structures $\calM_1,\ldots, \calM_m$, each with finitely many $\sim$-classes, such that $\calH=\calF\cup \bigcup_{i=1}^m age(\calM_i)$, where $\calF$ is a trivial hereditary $\calL$-property.  Since $\calH$ is minimal, we may assume $\calF=\emptyset$ (since deleting $\calF$ does not change the asymptotic speed of $\calH$).  Further, since $age(\calM_i)$ is a basic hereditary $\calL$-property for each $i\in [m]$, Corollary \ref{cor:basic} and $\calH=\bigcup_{i=1}^m age(\calM_i)$ implies there is $1\leq i\leq m$ such that $age(\calM_i)$ is also polynomial of degree $k$ (respectively exponential of degree $\ell$).  By  minimality, $\calH=age(\calM_i)$.  Since $\calH=age(\calM_i)$ is polynomial of degree $k$ (respectively exponential of degree $\ell\geq 2$), Corollary \ref{cor:polcount1} implies $\calM_i$ has one infinite $\sim$-class and exactly $k$ elements in finite $\sim$-classes (respectively $\ell$ infinite $\sim$-classes).  We now show further, that in the case when $\calH$ is exponential of degree $\ell \geq 2$, $\calM_i$ has no finite $\sim$-classes.  Indeed, suppose it did.  Let $\calM_i'$ be the substructure of $\calM_i$ obtained by deleting the finite $\sim$-classes. Then $age(\calM_i')$ is a strict subproperty of $\calH$ which is exponential of degree $\ell$ by Corollary \ref{cor:polcor}, a contradiction. This takes care of the forward directions of both (1) and (2).

Suppose for the converse that $\calM$ is a countably infinite $\calL$-structure with one infinite $\sim$-class and exactly $k\geq 0$ elements contained in a finite $\sim$-classes (respectively with $\ell$ infinite $\sim$-classes and no finite $\sim$-classes).  By Corollary \ref{cor:polcount1}, $age(\calM)$ is polynomial of degree $k$ (respectively exponential of degree $\ell$).  Suppose by contradiction there is is a strict subproperty $\calH'$ of $age(\calM)$ which is also polynomial of degree $k$ (respectively exponential of degree $\ell$).  

Corollary \ref{cor:polcount1} implies there are finitely many countably infinite $\calL$-structures $\calM_1,\ldots, \calM_m$, each with finitely many $\sim$-classes, such that $\calH'=\calF\cup \bigcup_{i=1}^m age(\calM_i)$, where $\calF$ is a trivial hereditary $\calL$-property.  Again, there must be some $i\in [m]$ so that $age(\calM_i)$ is itself polynomial of degree $k$ (respectively exponential of degree $\ell$).  By Corollary \ref{cor:polcount1}, $age(\calM_i)$ has one infinite $\sim$-class and exactly $k$ elements contained in a finite $\sim$-classes (respectively with $\ell$ infinite $\sim$-classes an no finite ones).  It is straightforward to check that because $age(\calM_i)\subseteq age(\calM)$, we must have that $\calM_i\models Th_{\forall}(\calM)$.  

Since $\calM_i\models Th_{\forall}(\calM)$, by fact (2) from the end of Section 1.1, there is some $\calM'\models Th(\calM)$ with $\calM_i\subseteq \calM'$.  By downward L\"owenheim-Skolem, there is a countable $\calM''$ with $\calM_i\subseteq \calM''\prec \calM'$.  Our assumptions on the structure of $\calM$ imply that $Th(\calM)$ is countably categorical, and thus $\calM''$ is isomorphic to $\calM$.  Consequently, we have shown $\calM$ has a substructure isomorphic to $\calM_i$.  This along with what we have shown about the structure of $\calM$ and $\calM_i$ imply that $\calM_i$ must in fact be isomorphic to $\calM$, contradicting that $age(\calM_i)\subsetneq age(\calM)$.

\end{proof}

\section{penultimate range}\label{sec:oscillate}

In this section we show that for each $r\geq 2$ there is a hereditary property of $r$-uniform hypergraphs whose speed oscillates between speeds close the lower and upper bounds of the penultimate range (case (3) of Theorem \ref{thm:mainthm}).  Our example is a straightforward generalization of one used in the graph case (see \cite{BBW2}). We include the full proofs for completeness. 

Throughout this section $r\geq 2$ is an integer and $\calG$ is the class of finite $r$-uniform hypergraphs.  We will use different notational conventions in this section, as it requires no logic.  For this section, a \emph{property} $\calP$ means a class of finite $r$-uniform hypergraphs closed under isomorphism. The \emph{speed} of a property $\calP$ is the function $n\mapsto |\calP_n|$.  We denote elements of $\calG$ as pairs $G=(V,E)$ where $V$ is the set of vertices of $G$ and $E\subseteq {V\choose r}$ is the set of edges. In this notation, we let $v(G)=|V|$ and $e(G)=|E|$. Given $U\subseteq V$, $G[U]$ is the hypergraph $(U,E\cap {U\choose r})$.  Given a hypergraph $H=(U, E')$, we write $H\subseteq G$ if $H$ is an induced subgraph of $G$, i.e. if $U\subseteq V$ and $H=G[U]$.  We begin by defining properties which we will use throughout the section.

\begin{Definition}
For $c\in \mathbb{R}^{\geq0}$, define
\begin{align*}
\calS^c=\{G\in \calG : e(G)\leq cv(G)\}\text{ and }
\calQ^c=\{G\in \calG : H\in \calS^c\text{ for all }H\subseteq G\}.
\end{align*}
\end{Definition}

Suppose $G=(V,E)$ is a finite $r$-uniform hypergraph. The \emph{density of $G$} is $\rho(G)=e(G)/v(G)$, and we say $G$ is \emph{strictly balanced} if for all $V'\subsetneq V$, $\rho(G[V'])<\rho(G)$.  The following theorem, proved by Matushkin in \cite{matushkin}, is a generalization to hypergraphs of results about strictly balanced graphs (see  \cite{RR, JGT:JGT3190100214}).

\begin{Theorem}[Matushkin \cite{matushkin}]\label{matushkin}
Suppose $r\geq 2$ is an integer and $c \in \mathbb{Q}^{\geq 0}$.  There exists a strictly balanced $r$-uniform hypergraph with density $c$ if an only if $c \geq \frac{1}{r-1}$ or $c=\frac{k}{1+k(r-1)}$ for some integer $k\geq 1$.
\end{Theorem}

Given an vertex set $V$ and a partition $\calP=\{P_1,\ldots, P_k\}$ of $V$, an \emph{$r$-matching compatible with $\calP$} is a set $E\subseteq {V\choose r}$ such that for every $e\in E$ and $1\leq i\leq k$, $|e\cap P_i|\leq 1$, and for every $e\neq e'\in E$, $e\cap e'=\emptyset$.  We say $\calP$ is an \emph{equipartition} of $V$ if $||P_i|-|P_j||\leq 1$ for all $1\leq i,j\leq k$.  The following is a straightforward generalization of Theorem 16 in  \cite{ZitoSch}.
 
\begin{Proposition}\label{thm:genQ}
For any constant $c\geq \frac{1}{r-1}$, $|\calS_n^c|=n^{(r-1)(c+o(1))n}=|\calQ^c_n|$. 
\end{Proposition}
\begin{proof}
For the upper bound bound, note that by definition, for large $n$, both $|\calQ_n^c|$ and $|\calS_n^c|$ are bounded above by the following.
\begin{align*}
\sum_{j=0}^{\lfloor cn\rfloor}{{n\choose r}\choose j}\leq cn{n^r\choose cn}\leq (cn+1)\Big(\frac{n^re}{cn}\Big)^{cn} =n^{(r-1)cn(1+o(1))}.
\end{align*}
For the lower bound, it suffices to show $|\calQ_n^c|\geq n^{(r-1)cn-o(n)}$.  Assume first $c$ is rational. Then by Theorem \ref{matushkin}, we can choose a finite, strictly balanced $r$-uniform hypergraph $H=(V,E)$ with density $c$. Without loss of generality, say $V=[t]$ for some $t\in \mathbb{N}^{>0}$.  Note $|E|=ct$ and $H\in \calQ^c$.  Suppose $n\gg t$ is sufficiently large, and choose an equipartition $\calP=\{W_1,\ldots, W_t\}$ of $[n]$.   For each $e\in E$, choose $E_e$ to be a maximal matching in $[n]$ compatible with $\calP$ and satisfying $E_e\subseteq \bigcup_{x\in e}W_x$.  Note the number of ways to choose $E_e$ is at least $(\lfloor n/t\rfloor !)^{r-1}$.

We claim $G:=([n],\bigcup_{e\in E}E_{e})\in \calQ^c_n$.  Fix $X\subseteq [n]$.  We show $e(G[X])\leq c|X|$.  For each $1\leq i\leq t$, let $X_i=X\cap W_i$ and $n_i=|X_i|$.  For each $1\leq u\leq n$, let $V_{u}=\{i\in [t]: n_i\geq u\}$, and let $H_u=H[V_{u}]$. Let $\ell=\max\{u: V_u\neq \emptyset\}$, and observe that $\ell\leq \lceil n/t\rceil$.  Observe that $|X|=\sum_{u=1}^{\ell}v(H_u)$, since 
$$
|X|=\sum_{u=1}^{\ell}u|V_u\setminus V_{u+1}|=\sum_{u=1}^{\ell}u(|V_u|-|V_{u+1}|)=|V_1|-|V_{\ell+1}|+\sum_{u=2}^{\ell}(u-(u-1))|V_u|=\sum_{u=1}^{\ell}v(H_u),
$$
 where the last equality uses that $V_{\ell+1}=\emptyset$. We now show that $e(G[X])=\sum_{u=1}^{\ell}e(H_u)$.  For each $1\leq u\leq \ell$, let $V_u'=V_u\setminus V_{u+1}$, and given $1\leq i_1,\ldots, i_r\leq \ell$, let 
 $$
 e(V_{i_1}',\ldots, V'_{i_r})=|\{\{a_1,\ldots, a_r\}\in E: a_1\in V_{i_1}',\ldots, a_r\in V_{i_r}'\}|.
 $$
 Then by construction,
 \begin{align*}
e(G[X])=\sum_{1\leq i_1\leq \ldots\leq i_r\leq \ell_X}i_1e(V'_{i_1},\ldots, V_{i_r}').
\end{align*}
On the other hand, for each $1\leq u\leq \ell$, $e(H_u)=\sum_{u\leq i_1\leq \ldots \leq i_r\leq \ell}e(V_{i_1}',\ldots,V_{i_r}')$, so 
$$
\sum_{u=1}^{\ell}e(H_u)=\sum_{u=1}^{\ell}\sum_{u\leq i_1\leq \ldots \leq i_r\leq \ell}e(V_{i_1}',\ldots,V_{i_r}')=\sum_{1\leq i_1\leq \ldots\leq i_r\leq \ell}i_1e(V'_{i_1},\ldots, V_{i_r}').
$$
Thus $e(G[X])=\sum_{u=1}^{\ell}e(H_u)$. Since $H$ is strictly balanced of density $c$, we know that for each $u\in [\ell]$, $e(H_u)\leq cv(H_u)$.  Combining these observations yields that
$$
e(G[X])=\sum_{u=\ell}^ne(H_u)\leq c\sum_{i=1}^\ell v(H_u)=c|X|.
$$
Thus $G\in \calQ_n^c$.  Clearly distinct choices for the set $\{E_e: e\in E\}$ yield distinct elements $([n], \bigcup_{e\in E}E_e)$ of $\calQ_n^c$. Since for each $e\in E$, the number of ways to choose $E_e$ is at least $(\lfloor n/t\rfloor!)^{r-1}$, this shows that 
$$
|\calQ^c_n|\geq \Big(\Big(\lfloor n/t\rfloor!\Big)^{r-1}\Big)^{|E|}=\Big(\lfloor n/t\rfloor!\Big)^{(r-1)ct}=n^{(r-1)cn-o(n)}.
$$

Assume now $c$ is irrational. Note that for all $c'\leq c$, $\calQ^{c'}\subseteq \calQ^c$. Thus by the calculations above, for all $\frac{1}{r-1}\leq c'< c$, where $c'\in \mathbb{Q}$, we have $|\calQ^c_n|\geq n^{(r-1)(c'-o(1))n}$. Clearly this implies $|\calQ^c_n|\geq n^{(r-1)(c-o(1))n}$.
\end{proof}
\begin{Definition}
Given an increasing (possibly finite) sequence $\nu=(\nu_1,\nu_2,\ldots )$ of natural numbers, let 
$$
\calP^{\nu,c}=\{G\in \calG: \text{ if }H\subseteq G\text{ and $v(H)=\nu_i$ for some $i$, then }e(H)\leq c\nu_i\}.
$$
\end{Definition}

\begin{Lemma}\label{lem:Lemma9}
Let $c\geq \frac{1}{r-1}$, $\epsilon>1/c$, and $\nu=(\nu_i)_{i\in \mathbb{N}}$ a sequence of natural numbers with $k=\sup\{\nu_i: i\in \mathbb{N}\}\in \mathbb{N}\cup \{\infty\}$.  Then
\begin{enumerate}
\item $|\calP^{\nu,c}_n|\geq n^{(r-1)(c+o(1))n}$ and $|\calP^{\nu,c}_n|= n^{(r-1)(c+o(1))n}$ whenever $n=\nu_i$ for some $i\in \mathbb{N}$,
\item if $k< \infty$ and $n$ is sufficiently large, then $|\calP_{\nu,c}^n|\geq 2^{n^{r-\epsilon}}$.
\end{enumerate}
\end{Lemma}
\begin{proof}
Note $\calQ^c\subseteq \calP^{c,\nu}$ so by Proposition \ref{thm:genQ}, $|\calP^{c,\nu}_n|\geq |\calQ^c_n| \geq n^{(r-1)(c+o(1))n}$.  When $n=\nu_i$ for some $i\in \mathbb{N}$, then by definition, $\calQ_n^c\subseteq \calP^{\nu,c}_n\subseteq \calS_n^c$. Consequently Proposition \ref{thm:genQ} implies $|\calP^{c,\nu}_{n}|=n^{(c(r-1)+o(1))n}$. This shows (1) holds.

Assume now $k<\infty$ and let $(k)=(1,\ldots, k)$. Note $\calP^{(k),c}\subseteq \calP^{\nu,c}$, so it suffices to show that for large enough $n$, $|\calP^{(k),c}_n|\geq 2^{n^{2-\epsilon}}$.  Choose $\delta$ satisfying $\epsilon>\delta>1/c$ and let $p=n^{-\delta}$.  Recall that $G_{n,p}$ is the random $r$-uniform hypergraph on vertex set $[n]$, with edge probability $p$.  We consider the probability $G_{n,p}\notin \calP^{(k),c}_n$, i.e. the probability that there is $H\subseteq G_{n,p}$ with $v(H)\leq k$ and $e(H)>cv(H)$.  Given $1\leq j\leq k$ and $S\in {[n]\choose j}$, let $X_S:G_{n,p}\rightarrow \mathbb{N}$ be the random variable defined by $X_S(G)=1$ if $e(G[S])>cj$.  Note $\mathbb{P}(X_S= 1)\leq \sum_{i=\lceil cj\rceil}^{j\choose r}{{j\choose r}\choose i}p^i(1-p)^{{j\choose r}-i}\leq C_jp^{cj}$ for some constant $C_j$ depending only on $j$, $r$ and $c$.  Therefore we have the following.
\begin{align*}
\mathbb{P}(G_{n,p}\notin \calP^{(k),c}_n)&\leq \sum_{j=1}^k\sum_{S\in {[n]\choose j}}\mathbb{P}(X_S= 1)\leq \sum_{j=1}^k{n\choose j}C_jp^{cj}\leq  \sum_{i=1}^k C_j(n^{1-\delta c})^j.
\end{align*}
Since $\delta c>1$, we have $1-\delta c<0$ and thus $\mathbb{P}(G_{n,p}\notin \calP^{(k),c}_n)\rightarrow 0$ as $n\rightarrow \infty$.  Thus we may choose $n_0$ so that $\mathbb{P}(G_{n,p}\notin \calP^{(k),c}_{n})<1/3$ for all $n\geq n_0$.  

Fix any $0<\epsilon_0<1/6$. If $n$ is sufficiently large, $\mathbb{P}(e(G_{n,p})< pN/2)\leq 1/2+\epsilon_0$, where $N={n\choose r}$.  Combining this with the above, we have shown that for all sufficiently large $n$, 
$$
\mathbb{P}(G_{n,p}\in \calP_n^{(k),c}\text{ and }e(G_{n,p})\geq pN/2)\geq 1/2-\epsilon_0-1/3>0.
$$
Thus for all sufficiently large $n$, there exists $G=([n],E)\in \calP_n^{(k),c}$ such that $e(G)\geq pN/2$.  Since $([n],E')\in \calP_n^{(k),c}$ for all $E'\subseteq E$, we have that
$$
|\calP^{(k),c}_{n}|\geq 2^{pN/2} =2^{n^{-\delta}{n\choose r}/2}\geq 2^{n^{r-\epsilon}},
$$
where the last inequality is because $n$ is sufficiently large, and $\epsilon<\delta$. 
\end{proof}


We now show that for any $c\geq 1/(r-1)$ and $\epsilon >1/c$, there is a property whose speed oscillates between $n^{nc(r-1)(1-o(1))}$ and $2^{n^{r-\epsilon}}$.

\begin{Theorem}\label{thm:oscilate}
Let $c\geq \frac{1}{r-1}$ and $\epsilon>1/c$. There exists sequences $\nu=(\nu_i)_{i\in \mathbb{N}}$ and $\mu=(\mu_i)_{i\in \mathbb{N}}$ where $\mu_i=\nu_i-1$ for all $i\in \mathbb{N}$ such that the following hold.
\begin{enumerate}
\item $|\calP^{\nu,c}_n|=n^{(r-1)(c+o(1))n}$ whenever $n=\nu_i$,
\item $|\calP^{\nu,c}_n|\geq 2^{n^{r-\epsilon}}$ whenever $n=\mu_i$,
\item $n^{(r-1)(c+o(1))n}\leq |\calP^{\nu,c}|<2^{n^{r-\epsilon}}$ if $n\neq \mu_i$.
\end{enumerate}
\end{Theorem}
\begin{proof}
Set $\nu_0=r+1$. Assume now $k\geq 0$ and suppose by induction we have chosen $\nu_0,\ldots, \nu_k$.  Let $\nu=(\nu_1,\ldots, \nu_k)$ and note by Lemma \ref{lem:Lemma9}, $|\calP^{\nu,c}_n|\geq 2^{n^{r-\epsilon}}$ for large enough $n$.  Choose $\mu_{k}>\nu_k$ minimal so that $|\calP^{\nu,c}_{\mu_{k}}|\geq 2^{\mu_{k+1}^{r-\epsilon}}$ and set $\nu_{k+1}=\mu_{k}+1$.
\end{proof} 

\noindent {\bf Proof of Theorem \ref{thm:oscilate1}}  Let $\calL=\{R(x_1,\ldots, x_r)\}$.  Given $c\geq \frac{1}{r-1}$ and $\epsilon>1/c$, let $\nu=(\nu_i)_{i\in \mathbb{N}}$ and $\mu=(\mu_i)_{i\in \mathbb{N}}$ be sequences as in Theorem \ref{thm:oscilate}.  Let $T_{\nu,c}$ be the universal theory of $\calP^{\nu,c}$ and observe that class of models of $T_{\nu,c}$ is a hereditary $\calL$-property $\calH$ such that for arbitrarily large $n$, $|\calH_n|=n^{(r-1)(c+o(1))n}$ and for arbitrarily large $n$, $|\calH_n|\geq 2^{n^{r-\epsilon}}$. 
 \qed

\bibliographystyle{amsplain}


\end{document}